\documentclass{mcom-l}

\usepackage{cleveref}
\usepackage{amssymb}
\usepackage{graphicx}
\usepackage{subfig}

\newtheorem{theorem}{Theorem}[section]
\newtheorem{lemma}[theorem]{Lemma}

\theoremstyle{definition}
\newtheorem{definition}[theorem]{Definition}
\newtheorem{example}[theorem]{Example}

\theoremstyle{remark}
\newtheorem{remark}[theorem]{Remark}

\newcommand{\Omdh}{{\widehat{\Omega}_{\delta}}}
\newcommand{\Omdc}{{\Omega_{\delta}^{c}}}

\numberwithin{equation}{section}

\crefname{equation}{}{}
\crefname{theorem}{Theorem}{Theorems}
\crefname{figure}{Figure}{Figures}
\crefname{table}{Table}{Tables}
\crefname{section}{Section}{Sections}
\crefname{example}{Example}{Examples}
\crefname{lemma}{Lemma}{Lemmas}
\crefname{remark}{Remark}{Remarks}

\begin{document}

\title[Error estimates of FEM for nonlocal problems]{Error estimates of finite element methods for nonlocal problems using exact or approximated interaction neighborhoods}



\author{Qiang Du}
\thanks{Department of Applied Physics and Applied Mathematics, and Data Science Institute, Columbia University, NY 10027, USA (qd2125@columbia.edu)}

\author{Hehu Xie}
\thanks{LSEC, NCMIS, Academy of Mathematics and Systems Science,
Chinese Academy of Sciences, Beijing 100190, \& School of Mathematical Sciences,
University of Chinese Academy of
Sciences, Beijing, 100049, China (hhxie@lsec.cc.ac.cn)}

\author{Xiaobo Yin}
\thanks{School of Mathematics and Statistics, and Key Laboratory of Nonlinear Analysis \& Applications (Ministry of Education), Central China Normal University, Wuhan 430079, China (yinxb@ccnu.edu.cn)}

\author{Jiwei Zhang}
\thanks{School of Mathematics and Statistics, and Hubei Key Laboratory of Computational Science, Wuhan University, Wuhan 430072, China (jiweizhang@whu.edu.cn)}


\date{}


\begin{abstract}
We study the asymptotic error between the finite element solutions of nonlocal models with a bounded interaction neighborhood and the exact solution of the limiting local model. The limit corresponds to the case when the horizon parameter, the radius of the spherical nonlocal interaction neighborhood of the nonlocal model,  and the mesh size simultaneously approach zero. Two important cases are discussed: one involving the original nonlocal models and the other for nonlocal models with polygonal approximations of the nonlocal interaction neighborhood. Results of numerical experiments are also reported to substantiate the theoretical studies. 
\end{abstract}

\maketitle


\section{Introduction} Nonlocal modeling has become popular in recent years among many applications \cite{andreu2010nonlocal,buades2010image,du2019nonlocal,silling2000reformulation,silling2010peridynamic}. As nonlocal integral operators are the key components of nonlocal modeling, their effective numerical integrations play important roles in both practice applications and numerical analysis. Our study here is devoted to this topic of wide interest. We focus on a model problem described below.

Let $\Omega\subset\mathbb{R}^{d}$ denote a bounded and open polyhedron. For
$u_{\delta}({\bf x}): \Omega\rightarrow \mathbb{R}$, the nonlocal operator
$\mathcal{L}_{\delta}$ on $u_{\delta}({\bf x})$ is defined as
\begin{equation}\label{Nolocal_Operator}
\mathcal{L}_{\delta}u_{\delta}({\bf x})=2\int_{\mathbb{R}^{d}}\left(u_{\delta}({\bf y})-u_{\delta}({\bf x})\right)\gamma_{\delta}({\bf x},{\bf y})d{\bf y}\quad \forall {\bf x}\in\Omega,
\end{equation}
where the nonnegative symmetric mapping $\gamma_{\delta}({\bf x},{\bf y}):
\mathbb{R}^{d}\times\mathbb{R}^{d} \rightarrow \mathbb{R}$ is called a kernel.
The operator $\mathcal{L}_{\delta}$ is regarded nonlocal since the value of
$\mathcal{L}_{\delta}u_{\delta}$ at a point $\bf x$ involves information about $u_{\delta}$
at points ${\bf y}\neq{\bf x}$. In this paper, we consider the following nonlocal Dirichlet volume-constrained diffusion problem
\begin{equation}\label{nonlocal_diffusion}
     \left \{
     \begin{array}{rll}
-\mathcal{L}_{\delta}u_{\delta}({\bf x})&=f_{\delta}({\bf x})& \mbox{on}\: \Omega, \\
u_{\delta}({\bf x})&=g_{\delta}({\bf x}) & \mbox{on} \:\Omdc,
     \end{array}
     \right .
\end{equation}
where $\Omdc=\{{\bf y}\in \mathbb{R}^d\setminus\Omega:\, \mbox{dist}({\bf y},\partial \Omega)<\delta\}$ denotes the interaction domain,
$f_{\delta}$ and $g_{\delta}$ are given functions. For convenience, we denote by $\Omdh=\Omega\cup\Omdc$. Since interactions often occur over finite distances in real-world applications, we only consider kernels having bounded support, i.e., $\gamma_{\delta}({\bf x},{\bf y})\neq 0$ only
if ${\bf y}$ is within a neighborhood of ${\bf x}$.  For this neighborhood, a popular practice is to choose
a spherical domain, that is, a Euclidean ball $B_{\delta}({\bf x})$ centered at ${\bf x}$ with a radius $\delta$, i.e.,
\begin{equation}\label{kernel_finite}
\mbox{for}\:\: {\bf x}\in\Omega:\:
\gamma_{\delta}({\bf x},{\bf y})=0,\: \forall {\bf y}\in\mathbb{R}^d\setminus B_{\delta}({\bf x}).
\end{equation}
Here $\delta$ is known as the horizon parameter or the interaction radius. Volume constraints (VCs) imposed on $\Omdc$
are natural extensions to the nonlocal case of boundary conditions (BCs) for differential equation problems. While the Dirichlet case is given in 
\eqref{nonlocal_diffusion} as an illustration, one can find discussions on other BCs, for example, nonlocal versions of Neumann and Robin BCs  in
\cite{du2012analysis,du2023nonlocal}.  Along with \eqref{kernel_finite},  the kernel is further assumed to be radial, i.e. $\gamma_{\delta}({\bf x},{\bf y})=\widetilde{\gamma}_{\delta}(|{\bf y}-{\bf x}|)$, with the following conditions being satisfied:
\begin{equation}\label{General_kernel}
     \left \{
     \begin{array}{l}
\widetilde{\gamma}_{\delta}(s)>
 0\:\: \mbox{for}\:\: 0<s< \delta, \\
\widetilde{\gamma}_{\delta}(s)\in L^{1}(0,\delta),\\
w_{d}\int_{0}^{\delta}s^{d+1}\widetilde{\gamma}_{\delta}(s)ds=d,
     \end{array}
     \right .
\end{equation}
where $w_{d}$ is the surface area of the unit sphere in $\mathbb{R}^d$. 
The radial symmetry of $\gamma_{\delta}({\bf x},{\bf y})$ matches with the choice of the spherical interaction neighborhood. Note that the first condition of \cref{General_kernel} means that the kernel is strictly positive for ${\bf y}$ inside $B_{\delta}({\bf x})$. 
The second condition is a simplifying assumption, which implies that $\mathcal{L}_{\delta}$ is a bounded operator in $L^2$. The third condition
of \cref{General_kernel} is set to ensure that if the operator $\mathcal{L}_{\delta}$
converges to a limit as $\delta\to 0$, then the limit is given by $\mathcal{L}_{0}=\Delta$,
the classical diffusion operator, see related discussions in
\cite{tian2014asymptotically,tian2020asymptotically}. To be specific, if the corresponding
local problem is defined as follows:
\begin{equation}\label{local_diffusion}
     \left \{
     \begin{array}{rll}
-\mathcal{L}_{0}u_{0}({\bf x})&=f_{0}({\bf x})& \mbox{on}\: \Omega, \\
u_{0}({\bf x})&=g_{0}({\bf x}) & \mbox{on} \:\partial\Omega.
     \end{array}
     \right.
\end{equation}
\begin{figure}[tbhp]
\centering
\vspace{-.75cm}
\captionsetup[subfigure]{captionskip=-18pt}\hspace{-1cm}
\subfloat[\rm regular]{\includegraphics[width=5cm]{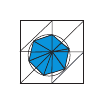}}\hspace{-1cm}
\subfloat[\rm nocaps]{\includegraphics[width=5cm]{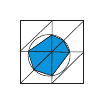}}\hspace{-1cm}
\subfloat[\rm approxcaps]{\includegraphics[width=5cm]{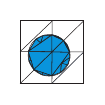}}
\vspace{-.5cm}\caption{Polygonally approximated balls, taken from \cite{du2022convergence}}\label{Fig:Polygon_Approx}
\end{figure}
Then the third condition of \cref{General_kernel} implies that as $\delta\rightarrow 0$, the nonlocal effect diminishes and  if the solutions of nonlocal
equations converge, the limit is given by the solution of a classical local differential equation. Such a limiting property
can serve to connect nonlocal and local models and is of great significance for
practical application, e.g., multiscale modeling and benchmark testing.  It is an important issue  of numerical analysis
to investigate the extent of preserving this limiting property on the discrete level for various discrete approximations. 
Motivated by the findings in
\cite{bobaru2010peridynamic,bobaru2009convergence,chen2011continuous,tian2013analysis}
that the numerical solutions based on the piecewise constant finite element methods (FEMs) fail to converge to a physically consistent local limit if the ratio of $\delta$ and $h$ (mesh size) is fixed, Tian and Du \cite{tian2014asymptotically} proposed 
the concept of asymptotically compatible (AC) schemes for the numerical approximations of a broad class of parametrized problems which, for nonlocal problems under consideration, are convergent discrete approximation schemes of the nonlocal models that preserve the correct local limiting
behavior. In other words, the numerical approximation given by an AC scheme can also reproduce
the correct local limiting solution as the horizon parameter and the mesh size approach
zero free of the relationship between them. They lead to consistent numerical results for benchmark problems without the need for a fine-tuning of model and discretization parameters, $\delta$ and $h$, and thus are regarded as more robust and more suitable for simulating problems involving nonlocal interactions on multiple scales such as in peridynamics \cite{askari2008peridynamics}. In \cite{tian2014asymptotically}, it was revealed that under very general assumptions on the problems, as long as the finite element space contains continuous piecewise linear functions, the Galerkin finite element approximation is always AC. However, the method using piecewise constant finite element is only a conditionally AC scheme under the condition that $\delta=o(h)$ \cite{bobaru2009convergence,chen2011continuous,tian2014asymptotically}. In
\cite{bobaru2009convergence}, one can find similar discussions, using numerical experiments, on the convergence in different  regimes of the parameters 
$\delta$ and $h$.

It is worth noting that most of the existing studies were carried out under the assumption that all the integrals
involved are computed exactly. The latter could be demanding, thus requiring careful consideration in practice \cite{zhang2016quadrature,pasetto2022efficient}. 
One important cause is due to the choice of the nonlocal interaction neighborhood of the operator
$\mathcal{L}_{\delta}$ defined in \cref{Nolocal_Operator} being confined to Euclidean balls. 
For finite element methods based on polyhedral meshes, dealing with intersections of such balls and mesh elements leads to considerable challenges in numerical implementation. As possible alternatives, the integrals may be approximated by truncation of the interaction neighborhoods, or by numerical integration, or a combination of both. For example, D'Elia et. al.  \cite{d2021cookbook} discussed several geometric approximations of the Euclidean balls in 2D (like \emph{regular}, \emph{nocaps} and \emph{approxcaps}
approximations, see \cref{Fig:Polygon_Approx}). They showed how such approximations, in combination
with quadrature rules, affect the discretization error. They also provided numerical
realization which shows that some geometric approximations preserve optimal accuracy in approximating the nonlocal solution under the assumption that the horizon parameter is fixed. However, when the AC property is concerned, the conclusion is different.  It is revealed in \cite{du2022convergence} that all
types of polygonal approximations considered in \cite{d2021cookbook} lose the desirable AC property even when they are used in the context of an AC scheme with exact integration. In other words, the numerical solutions of nonlocal problems with polygonal approximations may fail to converge to the right local limit. On the other hand, the AC property can hold conditionally which means they could be used with certain restrictions.  

The discussion in \cite{du2022convergence} was carried out on the continuum level, focused only on the approximations of the interaction neighborhoods. The conclusions are mostly qualitative.  Thus, quantitative analysis on both continuum and discrete levels remains an interesting issue, which provides the key motivation and the first objective of this study.  In fact, even for the FEM taking Euclidean balls as supports of nonlocal interaction kernels, there have been limited derivations of rigorous error estimates with respect to the horizon parameter and mesh size. Moreover, there has been no theoretical attempt to analyze the convergence of numerical approximations of peridynamic models where the spherical support of the kernel is polygonally approximated. The same observation can also be made in the case of meshfree methods when partial interaction volumes and geometric centers of intersecting regions are not precisely calculated. In contrast, it is worth mentioning that in \cite{du2016asymptotically} for Fourier spectral methods of nonlocal Allen-Cahn equation for 1D in space with periodic BCs,  error estimates with respect to horizon parameter and Fourier parameter have been rigorously derived for smooth local solutions.

As the second objective of this work, we study the convergence property of nonlocal solutions and their finite element approximations in both energy and $L^{2}$ norms. Recall the analog for the local PDE setting: the error in the energy norm is first derived, followed by the Aubin-Nitsche technique, which leads to the error estimate in the $L^{2}$ norm. Unfortunately, due to the possible lack of elliptic smoothing property,  the $L^{2}$ error estimate remained in the same order as the estimate in the nonlocal energy norm if the horizon parameter is fixed.  A natural question is how the Aubin-Nitsche theory behaves if the dependence of the horizon parameter is considered. 

In this work, the error estimates are derived using the following triangle inequality:
\begin{align}\label{Total_Err_En}
\left\|u_{\sharp}^{h}-\widetilde{u}_{0}\right\|_{\sharp}&\leq
\left\|u_{\sharp}-\widetilde{u}_{0}\right\|_{\sharp}+
\left\|u_{\sharp}-u_{\sharp}^{h}\right\|_{\sharp}:=E_{1}+E_{2},
\\ \label{Total_Err_L2}
\left\|u_{\sharp}^{h}-u_{0}\right\|_{L^{2}(\Omega)}&\leq
\left\|u_{\sharp}-u_{0}\right\|_{L^{2}(\Omega)}+
\left\|u_{\sharp}-u_{\sharp}^{h}\right\|_{L^{2}(\Omega)}.
\end{align}
Here, $\widetilde{u}_{0}$ is the $C^4$ extension of $u_0$ as defined in \cref{def:c4extension}, see also \cite{du2019uniform,yanguniform}.
Here the symbol $\sharp$ is used to identify the solution or the energy norm associated with one of the following two cases
\begin{equation}\label{two_kers}
\left\{\sharp=\delta; \: \: \sharp=(\delta,n_{\delta}) \right\}.
\end{equation}
The two cases identified by $\sharp$ correspond respectively to nonlocal solutions and the energy norms associated with interaction neighborhoods given by Euclidean balls and non-symmetric polygons with $n_{\delta}$ being the maximum number of polygons' sides. 

Concerning the energy norm $\left\|\cdot\right\|_{\sharp}$, the first part $E_{1}$ has the estimate
\begin{equation*}
E_{1}=\left \{
     \begin{array}{rl}
\|u_{\delta}-\widetilde{u}_{0}\|_{\delta}&\lesssim\delta^{(3+\mu)/2}, \\
\|u_{\delta,n_{\delta}}-\widetilde{u}_{0}\|_{\delta,n_{\delta}}&\lesssim
\delta^{(3+\mu)/2}+\delta^{-1}n_{\delta}^{-\lambda},
     \end{array}
     \right.
\end{equation*}
for $\lambda=2$ or $4$ with the value depending on the kernel, see \cref{lambda}. In this part, we also derive the estimate for
another interesting case, $\sharp=(\delta|n_{\delta})$, with the interaction neighborhoods given by symmetric polygons and the estimate of the form
\begin{equation*}
\|u_{\delta|n_{\delta}}-\widetilde{u}_{0}\|_{\delta}\lesssim
\delta^{(3+\mu)/2}+n_{\delta}^{-\lambda}.
\end{equation*}
One may observe that the error of $u_{\delta,n_{\delta}}$ against the local solution is one order lower than that of
$u_{\delta|n_{\delta}}$ with respect to $\delta$. To offer an explanation, we notice that for sufficiently smooth $\widetilde{u}_{0}$,
\begin{equation*}
\mathcal{L}_{\delta,n_{\delta}}\widetilde{u}_{0}({\bf x})=
F({\bf x})
+\sum_{i=1}^{2}\sigma_{n_{\delta}}^{ii,2}({\bf x})\partial_{ii}\widetilde{u}_{0}({\bf x})
+O\left(\delta^2\right),
\end{equation*}
where, due to the non-symmetric nature of the effective interaction domain, $F({\bf x})$ contains the odd moments up to the third-order and the second-order mixed moment of kernel $\gamma_{\delta,n_{\delta}}$
while the dominant term is given by the first-order moment. In contrast,
$\mathcal{L}_{\delta|n_{\delta}}\widetilde{u}_{0}({\bf x})$
does not contain any term like $F({\bf x})$, since the support of its kernel is symmetric 
with respect to ${\bf x}$. In fact, we have
$\left|F({\bf x})\right|
\lesssim\delta^{-1}n_{\delta}^{-\lambda}$,
which contributes to the extra error term of the error estimate in the non-symmetric case.

The second part $E_{2}$ is derived using the Cea's lemma in the nonlocal setting and 
the triangle inequality (see \cref{udstar-h} for $\sharp=\delta$, and \cref{udnstar-h} for $\sharp=\delta,n_{\delta}$) as
\begin{align*}
E_{2} &=\left\|u_{\sharp}-u_{\sharp}^{h}\right\|_{\sharp}
\leq\left\|u_{\sharp}-u_{\sharp}^{*}\right\|_{\sharp}+
\left\|u_{\sharp}^{*}-u_{\sharp}^{h}\right\|_{\sharp}
 \leq  
\left\|u_{\sharp}-u_{\sharp}^{*}\right\|_{\sharp}+
\left\|u_{\sharp}^{*}-I_{h} \widetilde{u}_{0} \right\|_{\sharp}\\
& \leq \left\|u_{\sharp}-u_{\sharp}^{*}\right\|_{\sharp}+
\left\|u_{\sharp}^{*}- u_{\sharp}\right\|_{\sharp}+
\left\| u_{\sharp} - \widetilde{u}_{0}\right\|_{\sharp}+
\left\| \widetilde{u}_{0} -
I_{h} \widetilde{u}_{0} \right\|_{\sharp}\\
& =  2\left\|u_{\sharp}-u_{\sharp}^{*}\right\|_{\sharp}+
\left\| \widetilde{u}_{0} -
I_{h} \widetilde{u}_{0} \right\|_{\sharp}  +
\left\| u_{\sharp} - \widetilde{u}_{0}\right\|_{\sharp}\\
&\leq C\delta^{-1/2}h^{2}+C\delta^{-1}\left\|\widetilde{u}_{0}
-I_{h}\widetilde{u}_{0}\right\|_{C\left(\Omdh\right)}+\left\|u_{\sharp}-\widetilde{u}_{0}\right\|_{\sharp}\\
&\leq C\delta^{-1/2}h^{2}+C\delta^{-1}h^{2}\left\|\widetilde{u}_{0}\right\|_{C^{2}\left(\Omdh\right)}
+\left\|u_{\sharp}-\widetilde{u}_{0}\right\|_{\sharp}\\
&\lesssim \delta^{-1}h^{2}+\left\|u_{\sharp}-\widetilde{u}_{0}\right\|_{\sharp}.
\end{align*}
Thus we obtain the total error estimate
\begin{equation*}
\left \{
     \begin{array}{rl}
\|u_{\delta}^{h}-\widetilde{u}_{0}\|_{\delta}&\lesssim\delta^{(3+\mu)/2}+\delta^{-1}h^{2}, \\
\|u_{\delta,n_{\delta}}^{h}-\widetilde{u}_{0}\|_{\delta,n_{\delta}}&\lesssim
\delta^{(3+\mu)/2}+\delta^{-1}h^{2}+\delta^{-\lambda-1}h^{\lambda},
     \end{array}
     \right.
\end{equation*}
under assumptions on the regularity of the local solution,
but not on that of the nonlocal solution, see \cref{Thm:Nonlocal-Fem-stand,Thm:Nonlocal-Fem-sym}.
A key contribution in this regard is the derivation of explicit error estimates with respect to the horizon parameter and the mesh size for finite element solutions of nonlocal problems under mild assumptions.
The sharpness of the error estimate in the energy norm is verified in \cref{sec:Numer_exp}. This further offers confidence in using the estimates to guide numerical simulations of problems like the numerical simulations of peridynamics. 

While the error estimate in the $L^{2}$ norm can be derived, via the Poincar\'{e}'s inequality, from the error in the energy norm, such a result is generally not sharp as shown in \cref{sec:Numer_exp}. Naturally, if a nonlocal analog of Aubin-Nitsche theory for FEM of the local problem (Aubin \cite{aubin1967behavior} and Nitsche \cite{nitsche1968kriterium}) could be developed, then improved estimates of the error in the $L^{2}$ norm  would be feasible, that is, an estimate of the form
\begin{equation*}
     \left \{
     \begin{array}{rl}
\|u_{\delta}^{h}-u_{0}\|_{L^{2}(\Omega)}&\lesssim \delta^{2}+h^{2}, \\
\|u_{\delta,n_{\delta}}^{h}-u_{0}\|_{L^{2}(\Omega)}&\lesssim 
\delta^{2}+h^{2}+\delta^{-\lambda}h^{\lambda},
     \end{array}
     \right .
\end{equation*}
might be expected. While the theoretical analysis is missing at the moment, it is verified numerically in \cref{sec:Numer_exp}. 
Based on the reported error estimates there, one can directly assess the convergence of the computed nonlocal solutions to the correct local limit for different choices of $\delta$ and $h$. 

The rest of the paper is organized as follows. In \cref{sec:Conv_d}, error estimates of
the nonlocal solutions with different kernels against the local counterpart ($E_{1}$) 
are derived in terms of horizon parameter and, for the second case of \cref{two_kers}, the maximum number of sides of polygons.
Next, we study the error between the nonlocal solutions and their finite element
approximations ($E_{2}$) in terms of horizon parameter and mesh size, and for the second case of \cref{two_kers}, the maximum number of polygons's sides in
\cref{sec:Conv_h}. In \cref{sec:ACtheory}, the error estimates in \cref{sec:Conv_d} are combined with that in \cref{sec:Conv_h} to obtain bounds of the error between
the nonlocal discrete solutions and the local exact solution. 
Results of numerical experiments are reported in \cref{sec:Numer_exp} to substantiate the theoretical studies. Finally, we give some concluding remarks in \cref{sec:concl}.

\section{Convergence of the nonlocal solutions to the local limit}\label{sec:Conv_d}
As in \cite{du2012analysis}, the nonlocal energy inner product, nonlocal energy norm, nonlocal energy space, and nonlocal constrained energy subspaces are defined by 
\begin{align*}
(u,v)_{\delta}:= & \int_{\Omdh}\int_{\Omdh}\left(u\left({\bf y}\right)-u\left({\bf x}\right)\right)\left(v({\bf y})-v({\bf x})\right)\gamma_{\delta}({\bf x},{\bf y})d{\bf y}d{\bf x},\:\:\|u\|_{\delta} :=  (u,u)_{\delta}^{1/2}, \\
V(\Omdh)    := & \left\{u\in L^2(\Omdh): u({\bf x})=g_{\delta}({\bf x})\:\mbox{on} \:\Omdc, \:\|u\|_{\delta}<\infty\right\},\\
V^{0}(\Omdh):= & \left\{u\in V(\Omdh): u({\bf x})=0\:\mbox{on} \:\Omdc\right\},\\
V^{c}(\Omdh):= & \left\{u\in V(\Omdh): u({\bf x})=0\:\mbox{on} \:\Omega\right\},
\end{align*}
respectively. For a bounded domain $D\subset\mathbb{R}^{d}$, the space of bounded and continuous functions is denoted by $C_{b}(D) := \left\{u\in C(D): u\: \mbox{is bounded on}\: \overline{D}\right\}$.
For any non-negative integer $n$, we define
\begin{equation*}
C_{b}^{n}(D) := \left\{u\in C_{b}(D): \forall\: \mbox{non-negative integer}\:
j\leq n, u^{(j)}\in C_{b}(D)\right\}.
\end{equation*}

\begin{definition}\label{def:c4extension}
For any function $u$ defined on $\Omega$, we let $\widetilde{u}$ be a $C^{n}$ extension of $u$ such that $\widetilde{u}\in C^{n}_{b}(\Omdh)$ and $\widetilde{u}|_{\Omega}=u$, see also Definition 4.1 in \cite{yanguniform}.
\end{definition}

\subsection{Convergence of the nonlocal solutions to the local limit}
In this subsection, {we investigate} for what kind of VCs,  the nonlocal solutions of
\cref{nonlocal_diffusion} converge to the local
solution of \cref{local_diffusion}, and with what asymptotic rate with respect to $\delta$.

\begin{theorem}\label{Thm:Nonlocal-local-stand}
Suppose $u_{0}\in C^{4}_{b}(\Omega)$ is the solution of the
local problem \eqref{local_diffusion}, the family of kernels $\{\gamma_{\delta}\}$
satisfies \cref{kernel_finite}, \cref{General_kernel}, and
\begin{equation}\label{K_delta_est}
G(\gamma_{\delta}):=\sup_{{\bf x}\in\Omdh}\int_{B_{\delta}({\bf x})\cap\Omdh}\gamma_{\delta}
({\bf x},{\bf y})d{\bf y}\lesssim \delta^{-2}.
\end{equation}
Let $u_\delta$ be the solution of the nonlocal problem \cref{nonlocal_diffusion}. If $\widetilde{u}_{0}$ is a $C^{4}$ extension of $u_{0}$ and
\begin{equation}\label{bc-rhs-est}
\left\|f_{\delta}-f_{0}\right\|_{C(\Omega)}=O\left(\delta^{2}\right),
\quad \left\|g_{\delta}-\widetilde{u}_{0}\right\|_{C(\Omdc)}=O\left(\delta^{2+\mu}\right),\:\mu=0,1,
\end{equation}
then it holds that
\begin{equation}\label{Conv_Stand}
\left\|u_{\delta}-u_{0}\right\|_{C(\Omega)}=O\left(\delta^{2}\right),\quad
\left\|u_\delta-\widetilde{u}_{0}\right\|_{\delta}=O\left(\delta^{(3+\mu)/2}\right).
\end{equation}
\end{theorem}
\begin{proof}
Since $\widetilde{u}_{0}\in C^{4}_{b}(\Omdh)$, a direct calculation leads to
\begin{align*}
-\mathcal{L}_{\delta}\widetilde{u}_{0}({\bf x})=-\mathcal{L}_{0}u_{0}({\bf x})
+O\left(\delta^2\right)|\widetilde{u}_{0}|_{4,\infty,\Omdh}=f_{0}({\bf x})+O\left(\delta^{2}\right),\: \forall\,{\bf x}\in \Omega.
\end{align*}
Thus, together with \cref{bc-rhs-est} it holds that
\begin{equation}\label{Basic_estimate}
     \left \{
     \begin{array}{ll}
-\mathcal{L}_{\delta}\left(u_\delta-\widetilde{u}_{0}\right)({\bf x})=
O\left(\delta^{2}\right), & {\bf x}\in \Omega,\\
u_\delta({\bf x})-\widetilde{u}_{0}({\bf x})=g_\delta({\bf x})-\widetilde{u}_{0}({\bf x})=O\left(\delta^{2+\mu}\right),& {\bf x}\in \Omdc.
     \end{array}
     \right.
\end{equation}
{The application of nonlocal maximum principle \cite{du2019uniform,tao2017nonlocal,yanguniform,you2020asymptotically} to \cref{Basic_estimate} produces} 
\begin{equation*}
\left\|u_\delta-u_{0}\right\|_{C(\Omega)}=\left\|u_\delta-\widetilde{u}_{0}\right\|_{C(\Omega)}
=O\left(\delta^{2}\right).
\end{equation*}
This, together with \cref{K_delta_est}, \cref{bc-rhs-est}, and
\cref{Basic_estimate}, leads to
\begin{align*}
&\left\|u_\delta-\widetilde{u}_{0}\right\|_{\delta}^{2}=-\int_{\Omega}\left(u_{\delta}({\bf x})
-u_{0}({\bf x})\right)\mathcal{L}_{\delta}\left(u_{\delta}({\bf x})-\widetilde{u}_{0}({\bf x})\right)d{\bf x}\nonumber\\
&-\int_{\Omdc}\left(g_{\delta}({\bf x})-\widetilde{u}_{0}({\bf x})\right)
\int_{\Omdh}\left(u_{\delta}({\bf y})-\widetilde{u}_{0}({\bf y})-g_{\delta}({\bf x})+\widetilde{u}_{0}({\bf x})\right)\gamma_{\delta}({\bf x},{\bf y})d{\bf y}d{\bf x}\nonumber\\
&\lesssim \delta^{4}+ |\Omdc|\cdot\left\|g_{\delta}-\widetilde{u}_{0}\right\|_{C(\Omdc)}
\cdot\left(\left\|u_{\delta}-u_{0}\right\|_{C(\Omega)}+
\left\|g_{\delta}-\widetilde{u}_{0}\right\|_{C(\Omdc)}
\right)\cdot G(\gamma_{\delta})\nonumber\\
&\lesssim \delta^{4}+\delta\cdot\delta^{2+\mu}\cdot \delta^{2}
\cdot\delta^{-2}\approx \delta^{3+\mu}.
\end{align*}
\end{proof}
\begin{remark}
Although the auxiliary solution $\widetilde{u}_{0}$ is required to be $C^{4}_{b}(\Omdh)$, the nonlocal solution itself $u_\delta$ is not subject to this restriction. As we will see in \cref{subsec:Numer:pert}, $u_\delta$ could be discontinuous across $\partial\Omega$.
\end{remark}

It is worth pointing out that,  nonlocal models with the second order convergence rate to the local limit in the $L^{\infty}$ norm have been studied in \cite{du2019uniform,lee2022second,tao2017nonlocal,yanguniform,you2020asymptotically}. 
In \cite{zhang2023second}, the authors considered nonlocal integral relaxations of local differential equations on a manifold, and proved the second-order convergence in the $H^{1}$ semi-norm to the local counterpart, although the nonlocal relaxation used there is different from what we discuss in this paper. 

In the rest of this paper, we take two-dimensional cases for illustration. A similar process of numerical analysis and the corresponding results can be readily extended to higher dimensions.

\subsection{Convergence of the nonlocal solutions with polygonal approximations of the {spherical} interaction neighborhoods}
To make the analysis more concise, we assume that $\widetilde{\gamma}_{\delta}$ has
a re-scaled form, that is,
\begin{equation}\label{rescaled_kernel}
\widetilde{\gamma}_{\delta} (s)
=\frac{1}{\delta^4} \gamma\left(\frac{s}{\delta}\right)
\end{equation}
for some nonnegative function $\gamma$ defined on $(0,1)$. Then by the normalization condition given in \cref{General_kernel} we have
\begin{equation*}
\int_{B_{1}({\bf 0})}\xi_{i}^{2}\gamma\left(|{\boldsymbol \xi}|_{2}\right)d{\boldsymbol \xi}=1,\:\: i=1,2.
\end{equation*}
Denote by
\begin{equation*}
\Phi(t)=\frac{1}{2}\int_{|{\boldsymbol \xi}|_{2}\leq t}|{\boldsymbol \xi}|_{2}^{2}\cdot\gamma(|{\boldsymbol
\xi}|_{2})d{\boldsymbol \xi}=\pi\int_{0}^{t}\rho^{3}\gamma(\rho)d\rho,
\quad
\Psi(t)=4\int_{0}^{t}\rho^{2}\gamma(\rho)d\rho.
\end{equation*}
Thus $\Phi(1)=1$ according to \cref{General_kernel}, and for $t\in(0,1]$,
\begin{equation*}
\Phi'(t)=\pi t^{3}\gamma(t),\:\: \Psi'(t)=4t^{2}\gamma(t).
\end{equation*}

\begin{figure}[tbhp]
\centering
\subfloat[$B_{\delta}({\bf x})$]{
\includegraphics[width=4cm]{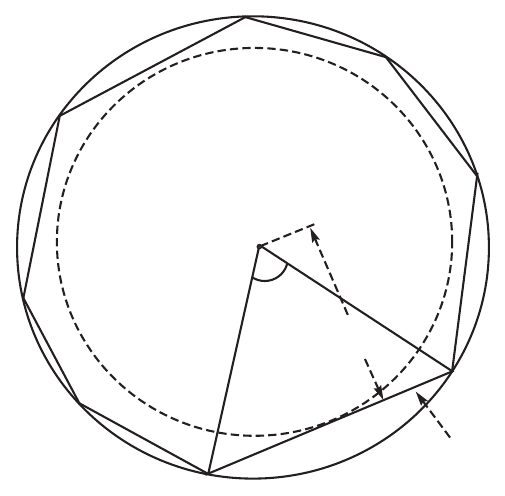}
\put(-60,60){${\bf x}$}\put(-43,35){$r_{\delta,{\bf x}}$}
\put(-20,5){$H_{n_{\delta}^{\bf x}}$}}
\hspace{1.5cm}
\subfloat[$B_{1}({\bf 0})$]{
\includegraphics[width=4.5cm]{Figure/Inscribed_Poly_Numer}
\put(-70,65){$O$}\put(-48,39){$\widehat{r}_{\delta,{\bf x}}$}
\put(-22,4){$\widehat{H}_{n_{\delta}^{\bf x}}$}}
\caption{Inscribed polygon in $B_{\delta}({\bf x})$ and its image by \cref{Affine_map}}\label{Fig:Inscribed}
\end{figure}

For any ${\bf x}\in\Omega$, an inscribed polygon of the circle $B_{\delta}({\bf x})$ is denoted
by $B_{\delta,n_{\delta}^{\bf x}}(\bf x)$, or simply $B_{\delta,n_{\delta}^{\bf x}}$,
where $n_{\delta}^{\bf x}$ denotes its number of sides. Furthermore, by \cref{rescaled_kernel} the rescaled polygon of $B_{\delta,n_{\delta}^{\bf x}}$ is defined as
\begin{equation}\label{Affine_map}
B_{1,n_{\delta}^{\bf x}}({\bf 0})=\left\{{\bf z}=\left({\bf y}-{\bf x}\right)/\delta:{\bf y}\in B_{\delta,n_{\delta}^{\bf x}}\right\}.
\end{equation}
For a given $\delta>0$ and a family of polygons $\{B_{\delta,n_{\delta}^{\bf x}}\}_{\bf x}$, we introduce the notation
\begin{equation*}
n_{\delta}=\sup_{{\bf x}\in\Omega}n_{\delta}^{\bf x},\quad n_{\delta, \inf}
=\inf_{{\bf x}\in\Omega}n_{\delta}^{\bf x}.
\end{equation*}
Denote by $r_{\delta,{\bf x}}$ the radius of the largest circle
(centered on ${\bf x}$) contained in $B_{\delta,n_{\delta}^{\bf x}}$, see (A) of \cref{Fig:Inscribed}, and $r(n_{\delta})=\inf\limits_{{\bf x}\in\Omega}r_{\delta,{\bf x}}$.

Let us define the weakly quasi-uniformity for a family of inscribed polygons which is
weaker than the quasi-uniformity introduced in \cite{du2022convergence}.
\begin{definition}
A family of inscribed polygons $\{B_{\delta,n_{\delta}^{\bf x}}\}$ is called weakly quasi-uniform if there exist two constants $C_{1}$ and $C_{2}>0$ such that $\forall\delta>0$, the following two bounds hold
\begin{equation*}
\sup_{{\bf x}\in\Omega}\frac{H_{n_{\delta}^{\bf x}}}
{2\sin\left(\pi/n_{\delta}^{\bf x}\right)}\leq C_{1}\delta, \quad\text{and}\quad \inf\limits_{{\bf x}\in\Omega}r_{\delta,{\bf x}}\geq C_{2}\delta
\end{equation*}
where $H_{n_{\delta}^{\bf x}}$ stands for
the length of the longest side of $B_{\delta,n_{\delta}^{\bf x}}$, see (A) of \cref{Fig:Inscribed}.
\end{definition}
For a weakly quasi-uniform family of inscribed polygons, there exists a constant
$C_{3}>0$ such that for all $\delta>0$, $n_{\delta}\leq C_{3} n_{\delta,\inf}$ holds. Denoted by 
\begin{equation*}
B'_{\delta,n_{\delta}^{\bf x}}=\{{\bf y}\in B_{\delta}({\bf x}):
{\bf x}\in B_{\delta,n_{\delta}^{\bf y}}\},
\end{equation*}
then the family of kernels 
\begin{equation*}
\gamma_{\delta,n_{\delta}}({\bf x},{\bf y})=\frac12 \gamma_{\delta}({\bf x},{\bf y})
\left(\chi_{B_{\delta,n_{\delta}^{\bf x}}}({\bf y})+\chi_{B'_{\delta,n_{\delta}^{\bf x}}}({\bf y})\right)
\end{equation*}
are symmetric with respect to ${\bf x}$ and ${\bf y}$, but not radial. We then define
a family of nonlocal operators
\begin{equation*}
\mathcal{L}_{\delta,n_{\delta}}u({\bf x})=
2\int_{\mathbb{R}^{2}}\left(u({\bf y})-u({\bf x})\right)
\gamma_{\delta,n_{\delta}}({\bf x},{\bf y})d{\bf y}\quad \forall {\bf x}\in\Omega,
\end{equation*}
and the corresponding family of nonlocal problems 
\begin{equation}\label{nonlocal_approx_sym}
     \left \{
     \begin{array}{rll}
-\mathcal{L}_{\delta,n_{\delta}}u_{\delta,n_{\delta}}({\bf x})&=f_{\delta}({\bf x})& \mbox{on}\: \Omega, \\
u_{\delta,n_{\delta}}({\bf x})&=g_{\delta}({\bf x}) & \mbox{on} \:\Omdc.
     \end{array}
     \right .
\end{equation}
The nonlocal energy inner product associated with $-\mathcal{L}_{\delta,n_{\delta}}$ and the
corresponding norm are defined as follows
\begin{align*}
(u,v)_{\delta,n_{\delta}}&:=\int_{\Omdh}\int_{\Omdh}\left(u({\bf y})-u({\bf x})\right)
\left(v({\bf y})-v({\bf x})\right)\gamma_{\delta,n_{\delta}}({\bf x},{\bf y})d{\bf y}d{\bf x},\\
\left\|u\right\|_{\delta,n_{\delta}}&:=(u,u)_{\delta,n_{\delta}}^{1/2}.
\end{align*}

To simplify the notation, let $\|u\|_{r(n_{\delta})}$ be the same as
$\|\cdot\|_{\delta}$  except for the replacement of $\delta$ by $r(n_{\delta})$.
As derived in \cite{du2022convergence} we have
\begin{equation}\label{equiv_norm_d_nd_s} 
\|u\|_{r(n_{\delta})}
\leq\|u\|_{\delta,n_{\delta}}\leq\|u\|_{\delta}.
\end{equation}

It is worth noting that for a fixed point $\bf x$, 
the support of $\gamma_{\delta,n_{\delta}}({\bf x},{\bf y})$ is 
not necessarily symmetric, so the integral of $({\bf y}-{\bf x})\gamma_{\delta,n_{\delta}}({\bf x},{\bf y})$ does not necessarily vanish. In order to inherit the symmetric-support property of the original kernel $\gamma_{\delta}$,
we define kernels 
\begin{equation*}
\gamma_{\delta|n_{\delta}}({\bf x},{\bf y})=\frac{1}{2}\gamma_{\delta}({\bf x},{\bf y})
\left(\chi_{B_{\delta,n_{\delta}^{\bf x}}}({\bf y})+\chi_{B^{T}_{\delta,n_{\delta}^{\bf x}}}({\bf y})\right),
\end{equation*}
with 
\begin{equation*}
B^{T}_{\delta,n_{\delta}^{\bf x}}=\{{\bf y}\in B_{\delta}({\bf x}):
2{\bf x}-{\bf y}\in B_{\delta,n_{\delta}^{\bf x}}\}.
\end{equation*}
Based on the definition of $\gamma_{\delta|n_{\delta}}$, we define a new family of nonlocal operators
\begin{equation*}
\mathcal{L}_{\delta|n_{\delta}}u({\bf x})=2\int_{\mathbb{R}^{2}}
\left(u({\bf y})-u({\bf x})\right)\gamma_{\delta|n_{\delta}}({\bf x},{\bf y})d{\bf y}\quad
\forall {\bf x}\in\Omega,
\end{equation*}
and the corresponding family of nonlocal problems 
\begin{equation}\label{nonlocal_approx}
     \left \{
     \begin{array}{rll}
-\mathcal{L}_{\delta|n_{\delta}}u_{\delta|n_{\delta}}({\bf x})&=f_{\delta}({\bf x})& \mbox{on}\: \Omega, \\
u_{\delta|n_{\delta}}({\bf x})&=g_{\delta}({\bf x}) & \mbox{on} \:\Omdc,
     \end{array}
     \right .
\end{equation}
where the support of the kernel is a {symmetric polygon for any fixed point {\bf x}}. However, the operator $-\mathcal{L}_{\delta|n_{\delta}}$ is not self-adjoint/symmetric so we could not directly define the corresponding inner product and norm.

We recall \cite{du2022convergence} to show that neither $\mathcal{L}_{\delta,n_{\delta}}\left( |{\bf x} |_{2}^{2}\right)$ nor
$\mathcal{L}_{\delta|n_{\delta}}\left( |{\bf x} |_{2}^{2}\right)$ converges to $\mathcal{L}_{0}\left( |{\bf x} |_{2}^{2}\right)$ if $n_{\delta}$ is uniformly
bounded, as $\delta\rightarrow 0$.
\begin{theorem}\label{Thm:Notconv}\cite[Theorem 3.1]{du2022convergence}
Assume $\left\{ B_{\delta,n_{\delta}^{\bf x}}\right\}$ is a family of polygons which satisfies
$B_{\delta,n_{\delta}^{\bf x}} \subset B_{\delta}({\bf x})$, $\forall \delta$, $\forall {\bf x} \in \Omega$. The family of kernels $\{\gamma_{\delta}\}$ satisfies the conditions \cref{kernel_finite}
and \cref{General_kernel}. If $n_{\delta}$ is uniformly
bounded as $\delta\rightarrow 0$, then both $\mathcal{L}_{\delta,n_{\delta}}\left( |{\bf x} |_{2}^{2}\right)$ and $\mathcal{L}_{\delta|n_{\delta}}\left( |{\bf x} |_{2}^{2}\right)$ do not converge to $\mathcal{L}_{0}\left( |{\bf x} |_{2}^{2}\right)$ for all ${\bf x}\in\Omega$.
\end{theorem}

Next, we want to establish the error estimate between $u_{\delta|n_{\delta}}$ and $\widetilde{u}_{0}$. To do that, we need a theorem of the maximum-principle type, which is important for the derivation of the estimate of $u_{\delta|n_{\delta}}-\widetilde{u}_{0}$,
once the estimate for $-\mathcal{L}_{\delta|n_{\delta}}\left(u_{\delta|n_{\delta}}-\widetilde{u}_{0}\right)$ is available. 

\begin{theorem}\label{Thm:Max_Princ}
Assume that $\gamma_{\delta}$ satisfies \cref{kernel_finite}
and the first two of \cref{General_kernel}.
If $f_{\delta}({\bf x})\leq 0$, for all ${\bf x}\in\Omega$, then
\begin{equation}\label{Max_Princ_Basic}
\sup_{{\bf x}\in\Omega}u_{\delta|n_{\delta}}({\bf x})\leq \sup_{{\bf x}\in\Omdc}g_{\delta}({\bf x}).
\end{equation}
\end{theorem}
\begin{proof}
Set $M=\sup\limits_{{\bf x}\in\Omega}u_{\delta|n_{\delta}}({\bf x})$.
We begin with the following case:
\begin{equation}\label{firstcase1}
f_{\delta}({\bf x})\leq -\eta<0,\: \forall {\bf x}\in\Omega,
\end{equation}
where $\eta$ is an arbitrary positive number. In this case we prove \cref{Max_Princ_Basic}
by contradiction.
If \cref{Max_Princ_Basic} does not hold. By the definition of $M$, we know that $u_{\delta|n_{\delta}}({\bf y})\leq M$, $\forall{\bf y}\in\Omega$ and there exists a point ${\bf x}_{0}\in\Omega$ such that
$u_{\delta|n_{\delta}}({\bf x}_{0})>M-\eta/4/G(\gamma_{\delta})$.
Then
\begin{align*}
f_{\delta}({\bf x}_{0})&=-\mathcal{L}_{\delta|n_{\delta}}u_{\delta|n_{\delta}}({\bf x}_{0})=2\int_{B_{\delta}({\bf x}_{0})}\left(u_{\delta|n_{\delta}}({\bf x}_{0})-u_{\delta|n_{\delta}}({\bf y})\right)\gamma_{\delta|n_{\delta}}({\bf x}_{0},{\bf y})d{\bf y}\\
&>2\int_{B_{\delta}({\bf x}_{0})}\left(M-u_{\delta|n_{\delta}}({\bf y})\right)
\gamma_{\delta|n_{\delta}}({\bf x}_{0},{\bf y})d{\bf y}-2\cdot\frac{\eta}{4G(\gamma_{\delta})}
\cdot G(\gamma_{\delta})\geq-\eta/2,
\end{align*}
which contradicts \cref{firstcase1}.

For the general case, i.e. $f_{\delta}({\bf x})\leq 0$ for all ${\bf x}\in\Omega$, we introduce the notation
\begin{equation*}
\widehat{r}_{\delta,{\bf x}}=r_{\delta,{\bf x}}/\delta,\:\:
\widehat{H}_{n_{\delta}^{\bf x}}=H_{n_{\delta}^{\bf x}}/\delta,\:\:
\hat{r}(n_{\delta})=r(n_{\delta})/\delta,\:\:
\hat{r}=\inf\limits_{\delta}\{\hat{r}(n_{\delta})\},\:\:
\eta^{*}=4\Phi(\hat{r}).
\end{equation*}
One can see (B) of \cref{Fig:Inscribed} for the description of $\widehat{r}_{\delta,{\bf x}}$ and $\widehat{H}_{n_{\delta}^{\bf x}}$. Since $\Omega$ is bounded, there exists a point ${\bf x}^{*}\in\Omega$ and a constant $R$ such that $\Omdh \subset B_{R}({\bf x}^{*})$.
Set $q({\bf x})=|{\bf x}-{\bf x}^{*}|_{2}^{2}$, 
then for all ${\bf x}\in\Omega$, it holds that
\begin{align}\label{eta-star}
\mathcal{L}_{\delta|n_{\delta}}q({\bf x})
&=2\int_{B_{\delta}({\bf x})}\left(q({\bf y})-q({\bf x})\right)\gamma_{\delta|n_{\delta}}({\bf x},{\bf y})d{\bf y}\nonumber\\
&=2\int_{B_{\delta}({\bf x})}|{\bf y}-{\bf x}|_{2}^{2}\gamma_{\delta|n_{\delta}}
({\bf x},{\bf y})d{\bf y}+
4\left({\bf x}-{\bf x}^{*}\right)\cdot\underbrace{\int_{B_{\delta}({\bf x})}({\bf y}-{\bf x})
\gamma_{\delta|n_{\delta}}({\bf x},{\bf y})d{\bf y}}_{={\bf 0}}\nonumber\\
&\geq 2\int_{B_{r(n_{\delta})}({\bf x})}|{\bf y}-{\bf x}|_{2}^{2}\cdot\gamma_{\delta}({\bf x},{\bf y})d{\bf y}\geq\eta^{*}.
\end{align}
With an arbitrary $\varepsilon>0$, $u_{\delta|n_{\delta}}+\varepsilon q$ is the solution of the following problem:
\begin{equation*}
     \left \{
     \begin{array}{rll}
-\mathcal{L}_{\delta|n_{\delta}}v({\bf x})&=f_{\delta}({\bf x})-\varepsilon\mathcal{L}_{\delta|n_{\delta}}q({\bf x}),& \mbox{on}\: \Omega, \\
v({\bf x})&=g_{\delta}({\bf x})+\varepsilon q({\bf x}),&  \mbox{on} \:\Omdc.
     \end{array}
     \right .
\end{equation*}
Since
$f_{\delta}({\bf x})-\varepsilon\mathcal{L}_{\delta|n_{\delta}}q({\bf x})\leq-\varepsilon\eta^{*}<0$,
by the previous argument, we have
\begin{equation*}
\sup_{{\bf x}\in\Omega}(u_{\delta|n_{\delta}}({\bf x})+\varepsilon q({\bf x}))\leq \sup_{{\bf x}\in\Omdc}(g_{\delta}({\bf x})+\varepsilon q({\bf x})).
\end{equation*}
Sending $\varepsilon$ to zero, we complete the proof.
\end{proof}

\begin{theorem}\label{Thm:Max_Princ_Bound}
Assume that $\gamma_{\delta}$ satisfies \cref{kernel_finite}
and the first two of \cref{General_kernel}. It holds
\begin{equation*}
\|u\|_{C(\Omega)}\leq \|u\|_{C(\Omdc)}+C_{4}\left\|\mathcal{L}_{\delta|n_{\delta}}
u\right\|_{C(\Omega)},
\end{equation*}
with
\begin{equation*}
C_{4}=\frac{2\|q\|_{C(\Omdc)}}{\eta^{*}}\leq\frac{R^{2}}{2\Phi(\hat{r})}.
\end{equation*}
\end{theorem}
\begin{proof}
Set for ${\bf x}\in\Omega$ that
\begin{equation*}
v^{\pm}({\bf x})=\pm u({\bf x})-\left\|\mathcal{L}_{\delta|n_{\delta}}u\right\|_{C(\Omega)}\cdot q({\bf x})/\eta^{*}.
\end{equation*}
By \cref{eta-star} we have 
\begin{equation*}
\mathcal{L}_{\delta|n_{\delta}}v^{\pm}({\bf x})=\pm\mathcal{L}_{\delta|n_{\delta}}u({\bf x})
-\left\|\mathcal{L}_{\delta|n_{\delta}}u\right\|_{C(\Omega)}\cdot\mathcal{L}_{\delta|n_{\delta}}q({\bf x})/\eta^{*}\leq 0, \:\:{\bf x}\in\Omega.
\end{equation*}
Applying \cref{Thm:Max_Princ}, we obtain
\begin{equation*}
\left\|v^{\pm}\right\|_{C(\Omega)}\leq\left\|u\right\|_{C(\Omdc)}
+\left\|\mathcal{L}_{\delta|n_{\delta}}u\right\|_{C(\Omega)}
 \left\|q\right\|_{C(\Omdc)}/\eta^{*}.
\end{equation*}
Thus we complete the proof.
\end{proof}

\begin{theorem}\label{Thm:Nonlocal-local-poly}
Suppose conditions of \cref{Thm:Nonlocal-local-stand} hold,
$\left\{ B_{\delta,n_{\delta}^{\bf x}}\right\}$ is a weakly quasi-uniform family of inscribed polygons. Let $u_{\delta|n_{\delta}}$ be the solution of the nonlocal problem
\cref{nonlocal_approx}. If $n_{\delta}\rightarrow \infty$ as $\delta\rightarrow 0$,
then it holds that
\begin{align}
\left\|u_{\delta|n_{\delta}}-u_{0}\right\|_{C(\Omega)}&=
O\left(\delta^{2}\right)+O\left(n_{\delta}^{-\lambda}\right),\label{SecOrder_Poly}\\
\left\|u_{\delta|n_{\delta}}-\widetilde{u}_{0}\right\|_{\delta}&=
O\left(\delta^{(3+\mu)/2}\right)+O\left(n_{\delta}^{-\lambda}\right),\label{SecOrder_Poly_Del}
\end{align}
with 
\begin{equation}\label{lambda}
 \lambda=    
 \left \{
     \begin{array}{l}
2, \:\:\mbox{if} \:\:\Phi'(1)\neq 0, \\
4, \:\:\mbox{if} \:\:\Phi'(1)=0, \:\:\Phi''(1)\neq 0.
     \end{array}
     \right.
\end{equation}
\end{theorem}
\begin{proof}
Since $\widetilde{u}_{0}\in C^{4}_{b}(\Omdh)$,  a direct calculation leads to
\begin{equation}\label{Taylor_Poly_Sym}
\mathcal{L}_{\delta|n_{\delta}}\widetilde{u}_{0}({\bf x})=
\sum_{i=1}^{2}\sigma_{n_{\delta}}^{i}({\bf x})\partial_{ii}\widetilde{u}_{0}({\bf x})
+O\left(\delta^2\right)\left|\widetilde{u}_{0}\right|_{4,\infty,\Omdh}
\end{equation}
with (the notation is slightly different from that in \cite{du2022convergence})
\begin{equation*}
\sigma_{n_{\delta}}^{i}({\bf x})=\int_{B_{\delta}({\bf x})}(y_{i}-x_{i})^{2}\gamma_{\delta|n_{\delta}}({\bf x},{\bf y})d{\bf y},\; i=1,2.
\end{equation*}
By the definition of $r(n_{\delta})$, we know for all ${\bf x}\in\Omega$ that
\begin{equation}\label{sigma_bound}
\int_{|{\bf z}|_{2}<r(n_{\delta})}z_{i}^2\widetilde{\gamma}_{\delta}\left(|{\bf z}|_{2}\right)d{\bf z}\leq
\sigma_{n_{\delta}}^{i}({\bf x})\leq\int_{|{\bf z}|_{2}<\delta}z_{i}^2
\widetilde{\gamma}_{\delta}\left(|{\bf z}|_{2}\right)d{\bf z}=\Phi(1)=1,
\end{equation}
where 
\begin{equation*}
\int_{|{\bf z}|_{2}<r(n_{\delta})}z_{i}^2\widetilde{\gamma}_{\delta}\left(|{\bf z}|_{2}\right)d{\bf z}
=\int_{|{\boldsymbol \xi}|_{2}<\hat{r}(n_{\delta})}\xi_{i}^2\gamma\left(|{\boldsymbol \xi}|_{2}\right)d{\boldsymbol \xi}
=\Phi\left(\hat{r}(n_{\delta})\right).
\end{equation*}
By weak quasi-uniformity of the inscribed polygons, we have for any ${\bf x}\in\Omega$,
\begin{align*}
\widehat{r}_{\delta,{\bf x}}^{\,2}=1-\left(\widehat{H}_{n_{\delta}^{\bf x}}/2\right)^{2}\geq 
1-C_{1}^{2}\sin^{2}\left(\pi/n_{\delta}^{\bf x}\right)
&\geq 1-C_{1}^{2}\sin^{2}\left(\pi/n_{\delta, \inf}\right)\\
&\geq 1-C_{1}^{2}\sin^{2}\left(C_{3}\pi/n_{\delta}\right).
\end{align*}
Taking the infimum on both sides of the above inequality, we have
\begin{equation*}
\hat{r}(n_{\delta})^{2}\geq 1-C_{1}^{2}\sin^{2}\left(C_{3}\pi/n_{\delta}\right),
\end{equation*}
which leads to
\begin{equation*}
1-\hat{r}(n_{\delta})\leq 1-\sqrt{1-C_{1}^{2}\sin^{2}\left(C_{3}\pi/n_{\delta}\right)}
=C_{1}^{2}C_{3}^{2}\pi^{2}/2/n_{\delta}^{2}+O\left(n_{\delta}^{-4}\right).
\end{equation*}
By the Taylor expansion of $\Phi(t)$ at the point $t=1$,
\begin{align*}
&1-\int_{|{\bf z}|_{2}<r(n_{\delta})}z_{i}^2\widetilde{\gamma}_{\delta}\left(|{\bf z}|_{2}\right)d{\bf z}
=\Phi(1)-\Phi\left(\hat{r}(n_{\delta})\right)\\
&=\left \{
     \begin{array}{rl}
\Phi'(1)\left(1-\hat{r}(n_{\delta})\right)+o\left(1-\hat{r}(n_{\delta})\right)
&=O\left(n_{\delta}^{-2}\right), \:\mbox{if} \:\Phi'(1)\neq 0, \\
-\Phi''(1)\frac{\left(1-\hat{r}(n_{\delta})\right)^{2}}{2}+
o\left(\left(1-\hat{r}(n_{\delta})\right)^{2}\right)
&=O\left(n_{\delta}^{-4}\right), \:\mbox{if} \:\Phi'(1)=0, \:\Phi''(1)\neq 0.
     \end{array}
     \right  .
\end{align*}
Thus, by \cref{lambda} and \cref{sigma_bound} we have
\begin{equation}\label{upperB}
1-\sigma_{n_{\delta}}^{i}({\bf x})\leq 1-\int_{|{\bf z}|_{2}<r(n_{\delta})}z_{i}^2\widetilde{\gamma}_{\delta}\left(|{\bf z}|_{2}\right)d{\bf z}
=O\left(n_{\delta}^{-\lambda}\right).
\end{equation}
Together with \cref{Taylor_Poly_Sym} and \cref{bc-rhs-est} it holds that
\begin{equation}\label{Oper_Conv}
     \left \{
     \begin{array}{ll}
-\mathcal{L}_{\delta|n_{\delta}}\left(u_{\delta|n_{\delta}}-\widetilde{u}_{0}\right)({\bf x})
=O\left(\delta^{2}\right)+O\left(n_{\delta}^{-\lambda}\right), & {\bf x}\in \Omega,\\
u_{\delta|n_{\delta}}({\bf x})-\widetilde{u}_{0}({\bf x})=g_\delta({\bf x})-\widetilde{u}_{0}({\bf x})=O\left(\delta^{2+\mu}\right),& {\bf x}\in \Omdc.
     \end{array}
     \right.
\end{equation}
A direct application of the maximum principle \cref{Thm:Max_Princ_Bound} to \cref{Oper_Conv} produces
\begin{equation*}
\left\|u_{\delta|n_{\delta}}-u_{0}\right\|_{C(\Omega)}
=O\left(\delta^{2}\right)+O\left(n_{\delta}^{-\lambda}\right)
\end{equation*}
This is \cref{SecOrder_Poly}, which, together with \cref{K_delta_est} and \cref{Oper_Conv}, leads to
\begin{align*}
&\left\|u_{\delta|n_{\delta}}-\widetilde{u}_{0}\right\|_{\delta}^{2}=
-\int_{\Omega}\left(u_{\delta|n_{\delta}}({\bf x})-u_{0}({\bf x})\right)
\mathcal{L}_{\delta}\left(u_{\delta|n_{\delta}}-\widetilde{u}_{0}\right)({\bf x})
d{\bf x}\\
&\;\; -\int_{\Omdc}\left(g_{\delta}({\bf x})-\widetilde{u}_{0}({\bf x})\right)
\int_{\Omdh}
\left(u_{\delta|n_{\delta}}({\bf y})-\widetilde{u}_{0}({\bf y})-g_{\delta}({\bf x})
+\widetilde{u}_{0}({\bf x})\right)\gamma_{\delta}({\bf x},{\bf y})d{\bf y}d{\bf x}\\
&\lesssim \left(\delta^{2}+n_{\delta}^{-\lambda}\right)^{2}+
|\Omdc|\cdot\left\|g_{\delta}-\widetilde{u}_{0}\right\|_{C(\Omdc)}\left(\left\|u_{\delta|n_{\delta}}-u_{0}\right\|_{C(\Omega)}+\left\|g_{\delta}-\widetilde{u}_{0}\right\|_{C(\Omdc)}
\right) G(\gamma_{\delta})\\
&\lesssim \left(\delta^{2}+n_{\delta}^{-\lambda}+\delta^{-1}\left\|g_{\delta}-\widetilde{u}_{0}\right\|_{C(\Omdc)}\right)
\left(\delta^{2}+n_{\delta}^{-\lambda}\right).
\end{align*}
Together with \cref{bc-rhs-est}, we get \cref{SecOrder_Poly_Del}. The proof is complete.
\end{proof}

By \cref{Thm:Nonlocal-local-poly}, if $n_{\delta}\rightarrow \infty$ as $\delta\rightarrow 0$, then $\{u_{\delta|n_{\delta}}\}$ converges to $\widetilde{u}_{0}$ in $\|\cdot\|_{C(\Omega)}$ (thus in $\|\cdot\|_{L^{2}(\Omega)}$) and $\|\cdot\|_{\delta}$ norms. This offers a contrast with the conclusion of \cref{Thm:Notconv}. 

Similar to \cref{Thm:Nonlocal-local-poly}, the error estimate for the nonlocal solutions with non-symmetric polygonal  
interaction neighborhoods against their local counterpart can be established. To state the result, we first introduce some notations. Let
\begin{align*}
\sigma_{n_{\delta}}^{i,1}({\bf x})&=\int_{\Omdh}(y_{i}-x_{i})
\gamma_{\delta,n_{\delta}}({\bf x},{\bf y})d{\bf y},\\
\sigma_{n_{\delta}}^{ij,2}({\bf x})&=\int_{\Omdh}(y_{i}-x_{i})(y_{j}-x_{j})
\gamma_{\delta,n_{\delta}}({\bf x},{\bf y})d{\bf y},\\
\sigma_{n_{\delta}}^{\boldsymbol{\alpha},3}({\bf x})&=\frac{1}{\boldsymbol{\alpha}!}
\int_{\Omdh}(y_{1}-x_{1})^{\alpha_{1}}(y_{2}-x_{2})^{\alpha_{2}}
\gamma_{\delta,n_{\delta}}({\bf x},{\bf y})d{\bf y}.
\end{align*}

\begin{theorem}\label{Thm:Nonlocal-local-poly-sym}
Suppose the conditions of \cref{Thm:Nonlocal-local-poly} hold.
Let $u_{\delta,n_{\delta}}$ be the solution of the nonlocal problem
\cref{nonlocal_approx_sym}. Under the condition 
\begin{equation}\label{consistent_cond}
\left|\sigma_{n_{\delta}}^{i,1}({\bf x})\right|\rightarrow 0,
\:\: i=1,2,\:\:  \mbox{as}\:\:  \delta\rightarrow 0,
\end{equation}
it holds that
\begin{align}
\left\|u_{\delta,n_{\delta}}-u_{0}\right\|_{C(\Omega)}=
O\left(\delta^{2}\right)+O\left(\delta^{-1}n_{\delta}^{-\lambda}\right),
\label{SecOrder_Poly-sym}\\
\left\|u_{\delta,n_{\delta}}-\widetilde{u}_{0}\right\|_{\delta,n_{\delta}}=
O\left(\delta^{(3+\mu)/2}\right)+O\left(\delta^{-1}n_{\delta}^{-\lambda}\right).
\label{SecOrder_Poly_Del-sym}
\end{align}
\end{theorem}

\begin{proof}
A direct calculation shows that
\begin{equation*}
\mathcal{L}_{\delta,n_{\delta}}\widetilde{u}_{0}({\bf x})=
F({\bf x})
+\sum\limits_{i=1}^{2}\sigma_{n_{\delta}}^{ii,2}({\bf x})\partial_{ii}\widetilde{u}_{0}({\bf x})+O\left(\delta^2\right),
\end{equation*}
where
\begin{equation*}
F({\bf x})=
2\sum_{i=1}^{2}\sigma_{n_{\delta}}^{i,1}({\bf x})\partial_{i}\widetilde{u}_{0}({\bf x})+
2\sigma_{n_{\delta}}^{12,2}({\bf x})\partial_{12}\widetilde{u}_{0}({\bf x})
+2\sum_{|\boldsymbol{\alpha}|=3}\sigma_{n_{\delta}}^{\boldsymbol{\alpha},3}({\bf x})
\partial_{\boldsymbol{\alpha}}\widetilde{u}_{0}({\bf x}).
\end{equation*}
Similar to the proof of \cref{upperB}, we have $1-\sigma_{n_{\delta}}^{ii,2}({\bf x})=O\left(n_{\delta}^{-\lambda}\right)$.
Thus, it holds that
\begin{equation}\label{Oper_Conv-sym}
     \left \{
     \begin{array}{ll}
-\mathcal{L}_{\delta,n_{\delta}}\left(u_{\delta,n_{\delta}}-\widetilde{u}_{0}\right)({\bf x})
=O\left(\delta^{2}\right)+O\left(n_{\delta}^{-\lambda}\right)+F({\bf x}), & {\bf x}\in \Omega,\\
u_{\delta,n_{\delta}}({\bf x})-\widetilde{u}_{0}({\bf x})=g_\delta({\bf x})-\widetilde{u}_{0}({\bf x})=O\left(\delta^{2+\mu}\right),& {\bf x}\in \Omdc.
     \end{array}
     \right.
\end{equation}
A direct application of nonlocal maximum principle similar to \cref{Thm:Max_Princ,Thm:Max_Princ_Bound} (but with the condition \cref{consistent_cond}) to \cref{Oper_Conv-sym} produces
\begin{equation}\label{Err_With_Fs}
\left\|u_{\delta,n_{\delta}}-u_{0}\right\|_{C(\Omega)}=
O\left(\delta^{2}\right)+O\left(n_{\delta}^{-\lambda}\right)+F_{s},
\end{equation}
where $F_{s}=\sup\limits_{{\bf x}\in\Omega}\left|F({\bf x})\right|$.
By the definition of $\gamma_{\delta,n_{\delta}}$, we have
\begin{align*}
&2\sigma_{n_{\delta}}^{i,1}({\bf x})=2\int_{B_{\delta}({\bf x})}(y_{i}-x_{i})
\gamma_{\delta,n_{\delta}}({\bf x},{\bf y})d{\bf y}\\
&=2\underbrace{\int_{B_{\delta}({\bf x})}(y_{i}-x_{i})
\gamma_{\delta}({\bf x},{\bf y})d{\bf y}}_{={\bf 0}}
-(\int_{B_{\delta}({\bf x})\setminus B_{\delta,n_{\delta}^{\bf x}}}
+\int_{B_{\delta}({\bf x})\setminus B'_{\delta,n_{\delta}^{\bf x}}})(y_{i}-x_{i})
\gamma_{\delta}({\bf x},{\bf y})d{\bf y}.
\end{align*}
Thus the following estimate holds
\begin{align*}
\left|\sigma_{n_{\delta}}^{i,1}({\bf x})\right|&\leq\frac{1}{2}
\int_{B_{\delta}({\bf x})\setminus B_{\delta,n_{\delta}^{\bf x}}
}+
\int_{B_{\delta}({\bf x})\setminus B'_{\delta,n_{\delta}^{\bf x}}}\left|y_{i}-x_{i}\right|
\gamma_{\delta}({\bf x},{\bf y})d{\bf y}\\
&\leq\int_{|{\bf z}|_{2}<\delta}|z_{i}|\widetilde{\gamma}_{\delta}\left(|{\bf z}|_{2}\right)d{\bf z}
-\int_{|{\bf z}|_{2}<r(n_{\delta})}|z_{i}|\widetilde{\gamma}_{\delta}\left(|{\bf z}|_{2}\right)d{\bf z}\\
&=\delta^{-1}\left[\Psi(1)-\Psi(\hat{r}(n_{\delta}))\right]
=O\left(\delta^{-1}n_{\delta}^{-\lambda}\right),
\end{align*}
where we use the Taylor expansion of $\Psi(t)$ at point $t=1$. Similarly,
\begin{equation*}
\left|\sigma_{n_{\delta}}^{12,2}({\bf x})\right|
=O\left(n_{\delta}^{-\lambda}\right),\:\:
\left|\sigma_{n_{\delta}}^{\boldsymbol{\alpha},3}({\bf x})\right|=O\left(\delta n_{\delta}^{-\lambda}\right).
\end{equation*} 
Then $F_{s}=O\left(\delta^{-1}n_{\delta}^{-\lambda}\right)$. Inserting this into \cref{Err_With_Fs}, we get \cref{SecOrder_Poly-sym} which, together with \cref{Oper_Conv-sym} and \cref{K_delta_est}, leads to
\begin{align*}
\left\|u_{\delta,n_{\delta}}-\widetilde{u}_{0}\right\|_{\delta,n_{\delta}}^{2}
&\lesssim \left(\delta^{2}+n_{\delta}^{-\lambda}+F_{s}\right)^{2}+
\left|\Omdc\right|\cdot\left\|g_{\delta}-\widetilde{u}_{0}\right\|_{C(\Omdc)}
\left(\delta^{2}+
n_{\delta}^{-\lambda}+F_{s}\right) G(\gamma_{\delta})\nonumber\\
&\lesssim \left(\delta^{2}+n_{\delta}^{-\lambda}+F_{s}\right)\cdot
\left(\delta^{1+\mu}+n_{\delta}^{-\lambda}+F_{s}\right).
\end{align*}
Thus, using the bound of $F_{s}$, we complete the proof.
\end{proof}

From \cref{Thm:Nonlocal-local-poly-sym}, we know that even when
$n_{\delta}\rightarrow \infty$, $\{u_{\delta,n_{\delta}}\}$ might not converge to $\widetilde{u}_{0}$ in $\|\cdot\|_{C(\Omega)}$ or $\|\cdot\|_{\delta,n_{\delta}}$ as $\delta\rightarrow 0$. The reason is that the support of
$\gamma_{\delta,n_{\delta}}({\bf x},{\bf y})$ is
not symmetric with respect to $\bf x$, so the integral of the product of the kernel with an odd function does not necessarily vanish. On the other hand, we could only prove the nonlocal maximum principle with the condition \cref{consistent_cond}. Thus, an additional condition is required to ensure the convergence in $\|\cdot\|_{C(\Omega)}$ or  $\|\cdot\|_{\delta,n_{\delta}}$: $F_{s}\rightarrow 0$ as $\delta\rightarrow 0$. However, one may observe the convergence in the $L^{2}$ norm numerically even without this condition, see numerical results in \cref{sec:Numer_exp} and additional discussions in \cref{Nitche}.

We could also prove lower bounds for $1-\sigma_{n_{\delta}}^{i}({\bf x})$ and $1-\sigma_{n_{\delta}}^{i,1}({\bf x})$ which play a helpful role in assessing whether the nonlocal operators converge to the local operator. The next theorem provides a lower bound for $1-\sigma_{n_{\delta}}^{i}({\bf x})$.

\begin{theorem}\label{Thm:Nonlocal-local-poly-lower}
Suppose the family $\{\gamma_{\delta}\}$
satisfies \cref{kernel_finite} and \cref{General_kernel}, while $\left\{ B_{\delta,n_{\delta}^{\bf x}}
\right\}$ is a family of
inscribed polygons. If $n_{\delta}\leq N$ for all $\delta$, then we have for $\Phi'(1)\neq 0$,
\begin{equation}\label{lowerB1}
1-\sigma_{n_{\delta}}^{i}({\bf x})
\geq L_{1}\cdot\sigma_{L,1}(\pi/N)/\pi,
\end{equation}
with
\begin{equation*}
L_{1}=\min_{\cos(\pi/N)\leq t\leq 1}\Phi'(t),\:\:
\sigma_{L,1}(s)=\frac{s}{2}+\frac{1}{4}\sin\left(2s
 \right)-\cos\left(s
 \right)
\log\left(\tan\left(\frac{s}{2} + \frac{\pi}{4}\right)\right),
\end{equation*}
and for $\Phi'(1)=0$, $\Phi''(1)\neq 0$,
\begin{equation}\label{lowerB2}
1-\sigma_{n_{\delta}}^{i}({\bf x})
\geq L_{2}\cdot\sigma_{L,2}(\pi/N)/\pi,
\end{equation}
with
\begin{equation*}
L_{2}=\min_{\cos(\pi/N)\leq t\leq 1}|\Phi''(t)|,\:\:
\sigma_{L,2}(s)=\frac{5}{8}\sin\left(2s
 \right)-\frac{s}{4}\cos\left(2s
 \right)-\cos\left(s
 \right)
\log\left(\tan\left(\frac{s}{2}+\frac{\pi}{4}\right)\right).
\end{equation*}
\end{theorem}
\begin{proof}
Consider the section of the longest side in $B_{1,n_{\delta}^{\bf x}}({\bf 0})$, which is composed 
of a triangle and its adjoining cap. Denote the cap by $C_{n_{\delta}^{\bf x}}$, then $C_{n_{\delta}^{\bf x}}\subset B_{1}({\bf 0})\setminus
B_{1,n_{\delta}^{\bf x}}({\bf 0})$, and 
\begin{align*}
&2\left(1-\sigma_{n_{\delta}}^{i}({\bf x})\right)=2\int_{B_{\delta}({\bf x})}(y_{i}-x_{i})^{2}\gamma_{\delta}({\bf x},{\bf y})d{\bf y}
-2\int_{B_{\delta}({\bf x})}(y_{i}-x_{i})^{2}\gamma_{\delta|n_{\delta}}({\bf x},{\bf y})d{\bf y}\\
&=\int_{B_{\delta}({\bf x})}(y_{i}-x_{i})^{2}\gamma_{\delta}({\bf x},{\bf y})\left(1-\chi_{B_{\delta,n_{\delta}^{\bf x}}}({\bf y})+1-\chi_{B^{T}_{\delta,n_{\delta}^{\bf x}}}({\bf y})\right)d{\bf y}\\
&\geq\int_{B_{\delta}({\bf x})}(y_{i}-x_{i})^{2}\gamma_{\delta}({\bf x},{\bf y})\left(1-\chi_{B_{\delta,n_{\delta}^{\bf x}}}({\bf y})\right)d{\bf y}
\geq \int_{C_{n_{\delta}^{\bf x}}}\xi_{i}^{2}\gamma\left(\left\|{\boldsymbol \xi}\right\|_{2}\right)d{\boldsymbol \xi}.
\end{align*}
According to different orientations, the value of
\begin{equation*}
\frac{1}{2}\int_{C_{n_{\delta}^{\bf x}}}\xi_{i}^{2}\gamma\left(\left\|{\boldsymbol \xi}\right\|_{2}\right)d{\boldsymbol \xi}
\end{equation*}
may differ. However, we can estimate a lower bound for all cases. For $\Phi'(1)\neq 0$,
\begin{align*}
&\frac{1}{2}\int_{C_{n_{\delta}^{\bf x}}}\xi_{i}^{2}\gamma\left(\left\|{\boldsymbol \xi}\right\|_{2}\right)d{\boldsymbol \xi}\geq \int_{0}^{\frac{\pi}{N}}\sin^{2}(\theta)\int_{\frac{\cos(\pi/N)}{\cos(\theta)}}^{1}\rho^{3}\gamma(\rho)d\rho d\theta\\
&=\frac{1}{\pi}\int_{0}^{\frac{\pi}{N}}\sin^{2}(\theta)\left[\Phi(1)-
\Phi\left(\frac{\cos(\pi/N)}{\cos(\theta)}\right)\right] d\theta
\geq\frac{L_{1}}{\pi}\int_{0}^{\frac{\pi}{N}}\sin^{2}(\theta)\left[1-
\frac{\cos(\pi/N)}{\cos(\theta)}\right] d\theta\\
& =L_{1}\cdot\left[\frac{\pi}{2N}+\frac{1}{4}\sin\left(\frac{2\pi}{N}
 \right)-\cos\left(\frac{\pi}{N}
 \right)
\log\left(\tan\left(\frac{\pi}{2N} + \frac{\pi}{4}\right)\right)\right]/\pi,
\end{align*}
which is \cref{lowerB1}. Using the similar argument and notice that for $\Phi'(1)=0$, $\Phi''(1)\neq 0$
\begin{equation*}
\Phi(1)-\Phi\left(\frac{\cos(\pi/n_{\delta})}{\cos(\theta)}\right)\\
\geq \frac{L_{2}}{2}\left[1-
\frac{\cos(\pi/n_{\delta})}{\cos(\theta)}\right]^{2},
\end{equation*}
then \cref{lowerB2} is gotten. Thus, we complete the proof.
\end{proof}

Note that this estimate for the lower bound of $1-\sigma_{n_{\delta}}^{i}({\bf x})$ in \cref{Thm:Nonlocal-local-poly-lower} might not be sharp in some cases. However, 
it could still play a helpful role in assessing whether the nonlocal operators converge to the local operator.

\cref{Thm:Nonlocal-local-poly-lower} offers a quantitative characterization of the lack of convergence stated in the \cref{Thm:Notconv}. In fact, it indicates that if $n_{\delta}$ is uniformly bounded as $\delta\rightarrow 0$, then $\{u_{\delta|n_{\delta}}\}$ may converge to $\widetilde{u}_{0}$ only for the case of a trivial solution. 

Similar results for $1-\sigma_{n_{\delta}}^{i,1}({\bf x})$ could be provided in the same manner, we omit that to save space for the paper.

\subsection{Examples of the kernel function}\label{subsec:kernel}
Here we list some popular kernels in general $d$-dimensional setting under the assumption that $\widetilde{\gamma}_{\delta}$ has a re-scaled form \cref{rescaled_kernel} while replacing $\delta^{4}$ (for $d=2$) by $\delta^{d+2}$. For more discussions
on the effects of the kernels on the nonlocal models, we refer to \cite{du2019nonlocal,seleson2011role}.

\textit{Type} 1. Constant kernel:
\begin{equation}\label{Con_Kernel_d}
\gamma(\rho)=\frac{d(d+2)}{w_{d}},\:\: 0\leq\rho\leq 1.
\end{equation}

\textit{Type} 2. Linear kernel:
\begin{equation*}
\gamma(\rho)=\frac{d(d+2)(d+3)}{w_{d}}(1-\rho),\:\: 0\leq\rho\leq 1.
\end{equation*}

\textit{Type} 3. Gaussian-like kernel:
\begin{equation*}
\gamma(\rho)=\frac{d}{C_{e}w_{d}}e^{-\rho^2},\:\: 0\leq\rho\leq
1,\:\: \mbox{with}\:\:  C_{e}=\int_{0}^{1}\tau^{d+1}e^{-\tau^2}d\tau.
\end{equation*}

\textit{Type} 4. Peridynamic kernel \cite{silling2005meshfree} for $d\geq 2$:
\begin{equation*}
\gamma(\rho)=\frac{d(d+1)}{w_{d}}\rho^{-1},\:\: 0< \rho\leq 1.
\end{equation*}

All four types of kernels shown above satisfy \cref{kernel_finite}, \cref{General_kernel}
and \cref{K_delta_est}. The  functions $\Phi(t)$ and $\Psi(t)$ corresponding to the linear
kernel (\textit{Type} 2) satisfy the condition
\begin{equation*}
\Phi'(1)=0, \:\: \Psi'(1)=0, \:\: \Phi''(1)\neq 0, \:\: \Psi'(1)\neq 0,
\end{equation*}
hence $\lambda=4$. The other three types of kernels satisfy
\begin{equation*}
\Phi'(1)\neq 0, \:\: \Psi'(1)\neq 0,
\end{equation*}
which leads to $\lambda=2$.

Note that our discussion so far involves only approximations of nonlocal operators on the continuum level and we have not incorporated the effect of finite dimensional discretizations of the operators. Next, we take the conforming DG (CDG) method as a representative setting to study the discretization error. Extensions to other types of numerical methods can be worked out analogously.

\section{Error estimate of the linear CDG method for the nonlocal problems
 with respect to the horizon parameter and the mesh size}\label{sec:Conv_h}
In this section, we review the linear CDG method proposed in \cite{du2019conforming} for
solving nonlocal problems with integrable kernels. We then study the corresponding discretization error with respect to the horizon parameter and the mesh size.
Note that in \cite{d2021cookbook} (where the continuous piecewise linear FEM is used) and \cite{du2019conforming}, the analysis is only concerned with the dependence on the mesh size.

\subsection{Brief review of the linear CDG method for nonlocal problems}\label{subsec:DG}
Now we briefly recall the basic steps
of the linear CDG method for solving the nonlocal problems \cref{nonlocal_diffusion} to make the discussion reasonably self-contained. Firstly, the variational form of \cref{nonlocal_diffusion} finds $u_{\delta}\in V(\Omdh)$ such that
\begin{equation}\label{nonlocal_var}
-\int_{\Omega}w({\bf x})\mathcal{L}_{\delta}u_{\delta}({\bf x})d{\bf x}=\int_{\Omega}f_{\delta}({\bf x})w({\bf x})d{\bf x},\: \forall w \in V^{0}(\Omdh).
\end{equation}

For a given triangulation $\mathcal{T}_{\delta}^{h}$ of $\Omdh$ that simultaneously
triangulates $\Omega$ (we call such a triangulation consistent), let 
\begin{equation*}
\Omega_{\delta}^{h}=\mathcal{T}_{\delta}^{h}\cap\overline{\Omega},\:\:
\Omega^{c,h}_{\delta}=\mathcal{T}_{\delta}^{h}\cap\overline{\Omdc}.
\end{equation*}
Assume that for any fixed $\delta$, $\mathcal{T}_{\delta}^{h}$
is quasi-uniform \cite{brenner2007mathematical} with respect to the mesh size $h$. The set of inner and boundary nodes of $\Omega_{\delta}^{h}$ are denoted by 
\begin{equation*}
N\!I=\{{\bf x}_j: j=1,\cdots, n_{1}\},\:\: \mbox{and}\:\: N\!B=\{{\bf x}_{n_{1}+j}: j=1,\cdots, n_{2}\},
\end{equation*}
respectively. The set of all nodes in $\Omega_{\delta}^{c,h}$ is denoted by
\begin{equation*}
N\!C=\{{\bf x}_{n_{1}+j}: j=1,\cdots, n_{2}+p\}.
\end{equation*}
Note that $N\!B \subset N\!C$. The continuous linear basis functions
defined on $\mathcal{T}_{\delta}^{h}$ are denoted by 
\begin{equation*}
\phi_{j}({\bf x}),\:\:j=1,2,\cdots, n_{1}+n_{2}+p.
\end{equation*}

The basis functions of $V_{\delta}^{0,h}$ which contains all piecewise linear
functions vanishing on $\mathcal{T}_{\delta}^{h}\!\setminus\!\Omega_{\delta}^{h}$ are defined as follows: for $j=1,2,\cdots,n_{1}+n_{2}$,
\begin{equation*}
     \widetilde{\phi}_{j}({\bf x})=\left \{
     \begin{array}{ll}
     \phi_{j}({\bf x})|_{\Omega_{\delta}^{h}}, & {\bf x}\in \Omega_{\delta}^{h},\\
     0, & {\bf x}\in \mathcal{T}_{\delta}^{h}\!\setminus\!\Omega_{\delta}^{h}.
     \end{array}
     \right .
\end{equation*}
The basis functions of $V^{c,h}_{\delta}$ which contains all piecewise linear
functions vanishing on $\mathcal{T}_{\delta}^{h}\!\setminus\!\Omega_{\delta}^{c,h}$ are
defined as follows: for $j=1,2,\cdots,p+n_{2}$,
\begin{equation*}
     \widetilde{\phi}_{j}^{c}({\bf x})=\left \{
     \begin{array}{ll}
     \phi_{n_{1}+j}({\bf x})|_{\Omega_{\delta}^{c,h}}, & {\bf x}\in \Omega_{\delta}^{c,h},\\
     0, & {\bf x}\in \mathcal{T}_{\delta}^{h}\!\setminus\!\Omega_{\delta}^{c,h}.
     \end{array}
     \right .
\end{equation*}
The linear CDG space is then defined as 
\begin{equation*}
V_{\delta}^{h}=V_{\delta}^{0,h}+g_{\delta}^{h},
\end{equation*}
where $g_{\delta}^{h}\in V^{c,h}_{\delta}$ is an approximation of the boundary data $g_{\delta}$.

The following \emph{conforming} property is satisfied, 
\begin{equation}\label{Set_conf}
V_{\delta}^{0,h}\subset V^{0}(\Omdh),\:\:\: V^{c,h}_{\delta}\subset V^{c}(\Omdh).
\end{equation}
An element of $V_{\delta}^{h}$ is continuous on $\Omega$ or $\Omdc$, but 
possibly discontinuous across $\partial\Omega$. The linear CDG approximation of \cref{nonlocal_var} finds $u_{\delta}^{0,h}\in V_{\delta}^{0,h}$ such that $\forall w_{h}\in V^{0,h}_{\delta}$,
\begin{align}\label{nonlocal_fem_exact}
-2\int_{\Omega}w_{h}({\bf x})\int_{B_{\delta}({\bf x})\cap\Omega}\left(u_{\delta}^{0,h}({\bf y})-u_{\delta}^{0,h}({\bf x})\right)\gamma_{\delta}({\bf x},{\bf y})d{\bf y}d{\bf x}\qquad\\
=\int_{\Omega}f_{\delta}({\bf x})w_h({\bf x})d{\bf x}+2\int_{\Omega}w_{h}({\bf x})\int_{B_{\delta}({\bf x})\cap\Omdc}\left(g_{\delta}^{h}({\bf y})-u_{\delta}^{0,h}({\bf x})\right)\gamma_{\delta}({\bf x},{\bf y})d{\bf y}d{\bf x}.\nonumber
\end{align}
Put together, we get $u_{\delta}^{h}=u_{\delta}^{0,h}+g_{\delta}^{h}$ as the final CDG approximation of $u_{\delta}$ on $\Omdh$.

In fact, the triangulations for $\Omdh$ are not required to
be consistent across the interface, since the continuity across $\partial\Omega$ of functions in the linear CDG space is not enforced (and there is no reason to assume the
continuity a priori). So $\overline{\Omega}$ and $\overline{\Omdc}$ could be triangulated separately by different mesh sizes (without any restriction when crossing the boundary) to get
$\Omega_{\delta}^{h}$ and $\Omega^{c,H}_{\delta}$. This offers more flexibility to implement compared to the use of consistent meshing. Hence a whole triangulation of $\Omdh$ is given as
\begin{equation*}
\mathcal{T}_{\delta}^{h,H}=\Omega_{\delta}^{h}\cup\Omega^{c,H}_{\delta},
\end{equation*}
which is called a non-consistent triangulation of $\Omdh$. Thus we define 
\begin{equation*}
V_{\delta}^{h,H}=V_{\delta}^{0,h}+g_{\delta}^{H},
\end{equation*}
where
$g_{\delta}^{H}\in V^{c,H}_{\delta}$ is an approximation of $g_{\delta}$. In this case,
the corresponding linear CDG approximation for \cref{nonlocal_var} finds
$u_{\delta}^{0,h}\in V_{\delta}^{0,h}$ such that \cref{nonlocal_fem_exact} (replacing $g_{\delta}^{h}$ by $g_{\delta}^{H}$) holds for any $w_{h}\in V^{0,h}_{\delta}$. 
And $u_{\delta}^{h,H}=u_{\delta}^{0,h}+g_{\delta}^{H}$ is the final CDG approximation of $u_{\delta}$. We assume that $\max\{h,H\}<\delta$ for convenience, which is particularly beneficial when dealing with issues related to VC, because 
$H<\delta$ can help avoid potential complications in describing the VC discretization. 

\subsection{The CDG approximation for the nonlocal problem
with spherical interaction neighborhoods}\label{subsec:Err_ud}
In \cite{du2019conforming}, the linear CDG method has been shown to yield an optimal
convergence rate with respect to the mesh size for some integrable kernels for fixed $\delta$. However,
the dependence of error estimate on $\delta$ has not been discussed so far. To analyze this dependence, we begin with the following lemma.
\begin{lemma}\label{Lemma:Norm_Equiv}
\cite{aksoylu2010results,Ponce2004An} Assume that the family of kernels $\{\gamma_{\delta}\}$ satisfies
the conditions \cref{kernel_finite} and \cref{General_kernel}. Then there exist
constants $\delta_{0}>0$, $C_{5}>0$ and $C_{6}>0$ such that for all $0<\delta\leq\delta_{0}$
and $v\in V^{0}(\Omdh)$, it holds
\begin{equation}\label{Norm_Equiv}
C_{5}\left\|v\right\|_{L^{2}(\Omega)}\leq \left\|v\right\|_{\delta}\leq C_{6}\delta^{-1}\left\|v\right\|_{L^{2}(\Omega)},
\end{equation}
where the constants $C_{5}$ and $C_{6}$ are all independent of $\delta$.
\end{lemma} 

Notice that the original lemma in \cite{aksoylu2010results} states that for general $d$ dimensional setting
\begin{equation*}
\widetilde{C}_{5}\delta^{d/2+1}\left\|v\right\|_{L^{2}(\Omega)}\leq \left\|v\right\|_{\delta}\leq
\widetilde{C}_{6}\delta^{d/2}\left\|v\right\|_{L^{2}(\Omega)},
\end{equation*}
under the assumption $\widetilde{\gamma}_{\delta}(\left\|{\bf z}\right\|_{2})=1$ instead of the last condition in \cref{General_kernel}, however, the two lemmas are equivalent.
The corresponding vector-valued version of the first inequality of \cref{Norm_Equiv}
is proven in \cite[Proposition 5.3]{mengesha2014bond}.

As stated in \cite[Remark 2.5]{aksoylu2010results}, the first inequality of \cref{Norm_Equiv} can be extended, as done in \cite[Proposition 2.5]{andreu2009nonlocal}, to the estimate
\begin{equation}\label{Coer_Nonhomo}
\widehat{C}_{5}\left\|v\right\|_{L^{2}\left(\Omdh\right)}\leq \left\|v\right\|_{\delta}+\int_{S_{-1}}|v({\bf x})|^{2}d{\bf x},
\end{equation}
for functions $v\in L^{2}(\Omdh)$ which do not necessarily vanish on $\Omdc$, where
\begin{equation*}
S_{-1}:=\{{\bf x}\in\Omdc: \mbox{dist}({\bf x},\partial\Omdh)\leq \delta/2\},
\end{equation*}
while the first inequality of \cref{Norm_Equiv} can be deduced from \cref{Coer_Nonhomo}
with the homogeneous Dirichlet condition on $\Omdc$ assumed. This result also coincides with the classical Poincar\'{e}'s inequality.

Similar to the proof for the error estimate in the $\|\cdot\|_{\delta}$ norm in \cref{Thm:Nonlocal-local-stand}, we can derive the relationship between the two norms $\left\|v\right\|_{\delta}$ and $\left\|v\right\|_{C(\Omega)}+\left\|v\right\|_{C(\Omdc)}$.
\begin{lemma}\label{Lemma:Norm_Equiv_inf}
Assume that the family of kernels $\{\gamma_{\delta}\}$ satisfies \cref{kernel_finite}, \cref{General_kernel} and \cref{K_delta_est}. Then
\begin{equation*}
\left\|v\right\|_{\delta}
\lesssim \delta^{-1}(\|v\|_{C(\Omega)}+\|v\|_{C(\Omdc)}), \quad \forall  v\in C_{b}(\Omega)\cap C_{b}(\Omdc).
\end{equation*}
\end{lemma}

To separate the influence due to the
discretization of VC, we introduce an intermediate problem which finds 
\begin{equation*}
u_{\delta}^{*}\in V^{0}(\Omdh)+g_{\delta}^{h}
\end{equation*}
such that
\begin{equation}\label{nonlocal_diffusion_interm}
-\mathcal{L}_{\delta}u_{\delta}^{*}({\bf x})=f_{\delta}({\bf x})\:\: \mbox{on}\:\: \Omega.
\end{equation}

\begin{theorem}\label{Thm:Convergence_ud}
Assume that the conditions of \cref{Thm:Nonlocal-local-stand} hold. Let $u_{\delta}$ and
$u_{\delta}^{h}$ be the solutions of
\cref{nonlocal_var} and \cref{nonlocal_fem_exact}, respectively.
Then
\begin{align}
\left\|u_{\delta}-u_{\delta}^{h}\right\|_{\delta}\lesssim \delta^{(3+\mu)/2}+\delta^{-1}h^{2},\label{Error_Numer_del}\\
\left\|u_{\delta}-u_{\delta}^{h}\right\|_{L^{2}(\Omega)}\lesssim \delta^{(3+\mu)/2}+\delta^{-1}h^{2}.\label{Error_Numer}
\end{align}
\end{theorem}

\begin{proof}
By \cref{nonlocal_diffusion} and \cref{nonlocal_diffusion_interm} one has
\begin{equation}\label{Homog-Ralation}
\mathcal{L}_{\delta}\left(u_{\delta}-u_{\delta}^{*}\right)({\bf x})=0,\: {\bf x}\in \Omega.
\end{equation}
The direct application of nonlocal maximum principle to \cref{Homog-Ralation} produces
\begin{equation}\label{aux_result}
\left\|u_{\delta}-u_{\delta}^{*}\right\|_{C(\Omega)}\lesssim \left\|g_{\delta}-g_{\delta}^{h}\right\|_{C(\Omdc)}
\lesssim h^{2}\left\|g_{\delta}\right\|_{C^{2}\left(\Omdc\right)}.
\end{equation}
By the generalized nonlocal Green's first identity (\cite{du2013nonlocal}) and \cref{Homog-Ralation} we get
\begin{align*}
&\left\|u_{\delta}-u_{\delta}^{*}\right\|_{\delta}^{2}=\!\int_{\Omdc}\!(u_{\delta}^{*}({\bf x})-u_{\delta}({\bf x}))\int_{\Omdh}\!\left(u_{\delta}({\bf y})-u_{\delta}^{*}({\bf y})-u_{\delta}({\bf x})+u_{\delta}^{*}({\bf x})\right)\gamma_{\delta}({\bf x},{\bf y})d{\bf y}d{\bf x}\nonumber\\
&=\int_{\Omdc}(g_{\delta}^{h}({\bf x})-g_{\delta}({\bf x}))\int_{\Omdh}\left(u_{\delta}({\bf y})-u_{\delta}^{*}({\bf y})-g_{\delta}({\bf x})+g_{\delta}^{h}({\bf x})\right)\gamma_{\delta}({\bf x},{\bf y})d{\bf y}d{\bf x}\nonumber\\
&\lesssim |\Omdc|\cdot\left\|g_{\delta}-g_{\delta}^{h}\right\|_{C(\Omdc)}\cdot\left(\left\|u_{\delta}-u_{\delta}^{*}\right\|_{C(\Omega)}+\left\|g_{\delta}-g_{\delta}^{h}\right\|_{C(\Omdc)}
\right)\cdot G(\gamma_{\delta})\nonumber\\
&\lesssim \delta\cdot h^{2}\left\|g_{\delta}\right\|_{C^{2}\left(\Omdc\right)}\cdot
h^{2}\left\|g_{\delta}\right\|_{C^{2}\left(\Omdc\right)}\cdot\delta^{-2}\approx \delta^{-1}h^{4}.
\end{align*}
Thus
\begin{equation}\label{aux_result_deltanorm}
\left\|u_{\delta}-u_{\delta}^{*}\right\|_{\delta}\lesssim\delta^{-1/2}h^{2}.
\end{equation}
Since $V_{\delta}^{0,h}\subset V^{0}(\Omdh)$ as pointed out in \cref{Set_conf}, then for all $w_{h} \in V_{\delta}^{0,h}$, it holds
\begin{equation*}
-\int_{\Omega}w_{h}({\bf x})\mathcal{L}_{\delta}u_{\delta}^{*}({\bf x})d{\bf x}=\int_{\Omega}w_{h}({\bf x})f_{\delta}({\bf x})d{\bf x},
\end{equation*}
which, together with \cref{nonlocal_fem_exact} and the nonlocal Green's first identity \cite{du2013nonlocal} leads to
\begin{equation*}
\left(u_{\delta}^{*}-u_{\delta}^{h},w_h\right)_{\delta}=0,\:\forall w_h \in V_{\delta}^{0,h}.
\end{equation*}
Thus we get the following estimate: for all $v_h \in V_{\delta}^{h}$,
\begin{equation*}
\left\|u_{\delta}^{*}-u_{\delta}^{h}\right\|_{\delta}^2 =\left(u_{\delta}^{*}-u_{\delta}^{h},u_{\delta}^{*}-u_{\delta}^{h}\right)_{\delta}
=\left(u_{\delta}^{*}-u_{\delta}^{h},u_{\delta}^{*}-v_h\right)_{\delta}
\leq \left\|u_{\delta}^{*}-u_{\delta}^{h}\right\|_{\delta} \left\|u_{\delta}^{*}-v_h\right\|_{\delta},
\end{equation*}
where the crucial relation $u_{\delta}^{h}-v_h\in V_{\delta}^{0,h}$ is used.
Then
\begin{equation}\label{cea_Euclidean}
\left\|u_{\delta}^{*}-u_{\delta}^{h}\right\|_{\delta}\leq\left\|u_{\delta}^{*}-v_h\right\|_{\delta}, \: \forall
v_h\in V_{\delta}^{h}.
\end{equation}
Let $v_h=I_{h}\widetilde{u}_{0}\in V_{\delta}^{h}$ be the piecewise linear interpolant of
$\widetilde{u}_{0}$ on $\mathcal{T}_{\delta}^{h}$.
By \cref{Conv_Stand}, \cref{Lemma:Norm_Equiv_inf}, \cref{aux_result_deltanorm}, and the error estimate for the Lagrangian interpolation, we have
\begin{align}
\left\|u_{\delta}^{*}-u_{\delta}^{h}\right\|_{\delta}
&\leq\left\|u_{\delta}^{*}-I_{h}\widetilde{u}_{0}\right\|_{\delta}\label{udstar-h}\\
&\leq \left\|\widetilde{u}_{0}-I_{h}\widetilde{u}_{0}\right\|_{\delta}+\left\|u_{\delta}^{*}-u_{\delta}\right\|_{\delta}+
\left\|u_{\delta}-\widetilde{u}_{0}\right\|_{\delta}\nonumber\\
&\leq C\delta^{-1}\left\|\widetilde{u}_{0}-I_{h}\widetilde{u}_{0}\right\|_{C\left(\Omdh\right)}
+C\delta^{-1/2}h^{2}+C\delta^{(3+\mu)/2}\nonumber\\
&\leq C\delta^{-1}h^{2}\left\|\widetilde{u}_{0}\right\|_{C^{2}\left(\Omdh\right)}+C\delta^{-1/2}h^{2}
+C\delta^{(3+\mu)/2}\nonumber\\
&\lesssim \delta^{(3+\mu)/2}+\delta^{-1}h^{2}.\nonumber
\end{align}
This, together with \cref{aux_result_deltanorm}, leads to \cref{Error_Numer_del}.

Since $u_{\delta}^{*}-u_{\delta}^{h}\in V^{0}(\Omdh)$, then by \cref{Lemma:Norm_Equiv}, we have
\begin{equation}\label{L2-energy}
\left\|u_{\delta}^{*}-u_{\delta}^{h}\right\|_{L^{2}\left(\Omega\right)}=\left\|u_{\delta}^{*}-u_{\delta}^{h}\right\|_{L^{2}\left(\Omdh\right)}
\lesssim \left\|u_{\delta}^{*}-u_{\delta}^{h}\right\|_{\delta}\lesssim \delta^{(3+\mu)/2}+\delta^{-1}h^{2}.
\end{equation}
This, together with \cref{aux_result}, leads to
\begin{equation*}
\left\|u_{\delta}-u_{\delta}^{h}\right\|_{L^{2}\left(\Omega\right)}\leq
\left\|u_{\delta}-u_{\delta}^{*}\right\|_{L^{2}\left(\Omega\right)}+
\left\|u_{\delta}^{*}-u_{\delta}^{h}\right\|_{L^{2}\left(\Omega\right)}
\lesssim \delta^{(3+\mu)/2}+\delta^{-1}h^{2},
\end{equation*}
this is \cref{Error_Numer}, the proof is complete.
\end{proof}

\begin{remark}\label{Nitche}
The convergence rates \cref{Error_Numer} is possibly not sharp, which is due to the
inequality \cref{L2-energy} used. As we know, when FEMs are used
to discretize a PDE, the Aubin-Nitsche technique often remains valid. Thus, the discretization error in the $L^2$ norm tends to exhibit a higher order than that in the $H^1$ semi-norm with respect to the mesh size. Although there is no proof of
its analog in the nonlocal problem setting with an integrable kernel (due to the lack of regularity pick-up), we are able to numerically observe this phenomenon from the numerical results in \cref{sec:Numer_exp}. That is, the $L^2$ norm of the discretization error exhibits a higher order with respect to the horizon parameter than the energy norm.
\end{remark}

\subsection{The CDG approximation for the nonlocal
problem with polygonal interaction neighborhoods}\label{subsec:Err_ud_nd_s}
The linear CDG approximation for the nonlocal problem \cref{nonlocal_approx_sym} finds
$u_{\delta,n_{\delta}}^{0,h}\in V_{\delta}^{0,h}$ such that for all $w_{h}\in V^{0,h}_{\delta}$,
\begin{align}\label{nonlocal_fem_nocap-nsym}
-2\int_{\Omega}w_{h}({\bf x})\int_{\Omega}
\left(u_{\delta,n_{\delta}}^{0,h}({\bf y})-u_{\delta,n_{\delta}}^{0,h}({\bf x})\right)
\gamma_{\delta,n_{\delta}}({\bf x},{\bf y})d{\bf y}d{\bf x}\\
=\int_{\Omega}f_{\delta}({\bf x})w_h({\bf x})d{\bf x}
+2\int_{\Omega}w_{h}({\bf x})\int_{\Omdc}
\left(g_{\delta}^{h}({\bf y})-u_{\delta,n_{\delta}}^{0,h}({\bf x})\right)
\gamma_{\delta,n_{\delta}}({\bf x},{\bf y})d{\bf y}d{\bf x},\nonumber
\end{align}
and $u_{\delta,n_{\delta}}^{h}=u_{\delta,n_{\delta}}^{0,h}+g_{\delta}^{h}$
is the final approximation of $u_{\delta,n_{\delta}}$. Note that since $\mathcal{T}_{\delta}^{h}$
is quasi-uniform, then 
\begin{equation}\label{dtoh}
n_{\delta}\sim\delta/h.
\end{equation}
We introduce the intermediate problem which finds $u_{\delta,n_{\delta}}^{*}\in V^{0}(\Omdh)+g_{\delta}^{h}$ such that
\begin{equation*}
-\mathcal{L}_{\delta,n_{\delta}}u_{\delta,n_{\delta}}^{*}
=f_{\delta}.
\end{equation*}
Using the similar argument in \cref{Thm:Convergence_ud}, we have the following error estimate.
\begin{theorem}\label{Thm:Convergence_appr-sym}
If the conditions of \cref{Thm:Nonlocal-local-poly-sym} hold. We have 
\begin{align}
\|u_{\delta,n_{\delta}}-u_{\delta,n_{\delta}}^{h}\|_{\delta,n_{\delta}}\lesssim
\delta^{(3+\mu)/2}+\delta^{-1}h^{2}+\delta^{-\lambda-1}h^{\lambda},
\label{Error_Numer_appr_del-sym}\\
\|u_{\delta,n_{\delta}}-u_{\delta,n_{\delta}}^{h}\|_{L^{2}\left(\Omega\right)}\lesssim
\delta^{(3+\mu)/2}+\delta^{-1}h^{2}+\delta^{-\lambda-1}h^{\lambda}.
\label{Error_Numer_appr-sym}
\end{align} 
\end{theorem}
\begin{proof}
Similar to the derivation in \cref{Thm:Convergence_ud}, we have
\begin{align}
\|u_{\delta,n_{\delta}}-u_{\delta,n_{\delta}}^{*}\|_{C(\Omega)}&\lesssim
\|g_{\delta}-g_{\delta}^{h}\|_{C(\Omdc)}\leq Ch^{2}\|g_{\delta}\|_{C^{2}\left(\Omdc\right)},\label{aux_result_appr}\\
\|u_{\delta,n_{\delta}}-u_{\delta,n_{\delta}}^{*}\|_{\delta,n_{\delta}}&\leq C\delta^{-1/2}h^{2}.\label{aux_result_ndeltanorm_noconstr}
\end{align}
Since $n_{\delta}\rightarrow\infty$ as $\delta\rightarrow 0$,
by \cref{equiv_norm_d_nd_s}, \cref{Lemma:Norm_Equiv_inf} and \cref{SecOrder_Poly_Del-sym}, we have
\begin{align}
&\|u_{\delta,n_{\delta}}^{*}-u_{\delta,n_{\delta}}^{h}\|_{\delta,n_{\delta}}\leq
\|u_{\delta,n_{\delta}}^{*}-I_{h}\widetilde{u}_{0}\|_{\delta,n_{\delta}}\label{udnstar-h}\\
&\leq\|u_{\delta,n_{\delta}}^{*}-u_{\delta,n_{\delta}}\|_{\delta,n_{\delta}}
+\|\widetilde{u}_{0}-I_{h}\widetilde{u}_{0}\|_{\delta,n_{\delta}}+
\|u_{\delta,n_{\delta}}-\widetilde{u}_{0}\|_{\delta,n_{\delta}}\nonumber\\
&\leq C\delta^{-1/2}h^{2}+C\delta^{-1}\|\widetilde{u}_{0}
-I_{h}\widetilde{u}_{0}\|_{C\left(\Omdh\right)}+\|u_{\delta,n_{\delta}}
-\widetilde{u}_{0}\|_{\delta,n_{\delta}}\nonumber\\
&\leq C\delta^{-1/2}h^{2}+C\delta^{-1}h^{2}\|\widetilde{u}_{0}\|_{C^{2}\left(\Omdh\right)}
+C\delta^{-1}n_{\delta}^{-\lambda}+Cn_{\delta}^{-\lambda}+C\delta^{(3+\mu)/2}\nonumber\\
&\lesssim \delta^{(3+\mu)/2}+\delta^{-1}h^{2}+\delta^{-1}n_{\delta}^{-\lambda}.\nonumber
\end{align}
This, together with \cref{aux_result_ndeltanorm_noconstr} and \cref{dtoh}, leads to \cref{Error_Numer_appr_del-sym}.
By the uniform Poincar\'{e}'s inequality for the norm $\|\cdot\|_{\delta}$, \cref{equiv_norm_d_nd_s} and \cref{aux_result_appr}, we get \cref{Error_Numer_appr-sym}.
\end{proof}

\section{Error estimates between the nonlocal discrete solutions and the local exact solution}\label{sec:ACtheory}
We now combine the error estimates in \cref{sec:Conv_d} and \cref{sec:Conv_h}
to derive the error estimate between the nonlocal discrete solutions and the local exact solution. 

First, by combining \cref{Thm:Nonlocal-local-stand} and \cref{Thm:Convergence_ud} we get the following theorem.
\begin{theorem}\label{Thm:Nonlocal-Fem-stand}
Suppose $u_{0}\in C^{4}_{b}(\Omega)$ is the solution of the
local problem \cref{local_diffusion}, the family of kernels $\{\gamma_{\delta}\}$
satisfies \cref{kernel_finite,General_kernel,K_delta_est}.
If $\widetilde{u}_{0}$ is a $C^{4}$ extension of $u_{0}$, $u_{\delta}^{h}$ is the linear CDG approximation of the
nonlocal problem \cref{nonlocal_diffusion} under the condition \cref{bc-rhs-est},
then it holds that
\begin{align}
\left\|\widetilde{u}_{0}-u_{\delta}^{h}\right\|_{\delta}
&=\left\|\widetilde{u}_{0}-u_{\delta}+u_{\delta}-u_{\delta}^{h}\right\|_{\delta}\lesssim
\delta^{(3+\mu)/2}+\delta^{-1}h^{2},\label{udh_del}\\
\left\|u_{0}-u_{\delta}^{h}\right\|_{L^{2}\left(\Omega\right)}
&=\left\|u_{0}-u_{\delta}+u_{\delta}-u_{\delta}^{h}\right\|_{L^{2}\left(\Omega\right)}\lesssim
\delta^{(3+\mu)/2}+\delta^{-1}h^{2}.\label{udh_L2}
\end{align}
\end{theorem}

To make the analysis related to polygonal approximation concrete, we take the constant kernel \cref{Con_Kernel_d} as an illustrative example ($\lambda=2$) and adopt \emph{nocaps} as an approximation of the spherical neighborhood, that is, an approximation using the inscribed polygon of the 
Euclidean disc. We note that the analysis can be extended to other types of kernels listed in \cref{subsec:kernel} and other types of polygonal approximations. Since $\lambda=2$ and \cref{dtoh}, together with \cref{Thm:Nonlocal-local-poly-sym} and
\cref{Thm:Convergence_appr-sym} we get the following theorem.
\begin{theorem}\label{Thm:Nonlocal-Fem-sym}
Suppose $u_{0}\in C^{4}_{b}(\Omega)$ is the solution of the
local problem \cref{local_diffusion}, the family of kernels $\{\gamma_{\delta}\}$
satisfies \cref{kernel_finite,General_kernel,K_delta_est}. $\left\{ B_{\delta,n_{\delta}^{\bf x}}\right\}$ is a weakly quasi-uniform family of inscribed polygons. $\widetilde{u}_{0}$ is a $C^{4}$ extension of $u_{0}$ and $u_{\delta,n_{\delta}}^{h}$ is the linear CDG approximation of the nonlocal problem
\cref{nonlocal_approx_sym} under the condition \cref{bc-rhs-est}. If \cref{consistent_cond},
then it holds that
\begin{align}
\|\widetilde{u}_{0}-u_{\delta,n_{\delta}}^{h}\|_{\delta,n_{\delta}}
&\lesssim\delta^{(3+\mu)/2}+\delta^{-1}h^{2}+\delta^{-\lambda-1}h^{\lambda}
\sim\delta^{(3+\mu)/2}+\delta^{-3}h^{2}
,\label{Error_Total_Del-sym}\\
\|u_{0}-u_{\delta,n_{\delta}}^{h}\|_{L^{2}\left(\Omega\right)}
&\lesssim
\delta^{(3+\mu)/2}+\delta^{-1}h^{2}+\delta^{-\lambda-1}h^{\lambda}\sim\delta^{(3+\mu)/2}+\delta^{-3}h^{2}.\label{Error_Total-sym}
\end{align}
\end{theorem}

\section{Numerical experiment}\label{sec:Numer_exp}
We first consider in \cref{subsec:Numer:exact} the nonlocal problems with the exact right-hand side (RHS) function and VC.
Hence the error induced in \cref{Thm:Nonlocal-local-stand} vanishes.
We report the convergence results of \emph{exactcaps} ($u_{\delta}^{h}$) and \emph{nocaps} ($u_{\delta,n_{\delta}}^{h}$) solutions for three cases of parameter setting: fixed horizon parameter, fixed ratio and power function between the horizon parameter and the mesh size. They corresponds to $m$-convergence, $\delta$-convergence, and $\delta m$-convergence for numerical methods of peridynamics defined in \cite{bobaru2009convergence}, respectively. Then in \cref{subsec:Numer:pert} the nonlocal problems with a perturbed RHS function and VC are discussed.

\subsection{Nonlocal problems with an exact RHS function and VC}\label{subsec:Numer:exact}
\begin{example}\label{Example:Continuous}
We consider the nonlocal problem \cref{nonlocal_diffusion} on the domain $\Omega=(0,1)^{2}$
with the family of kernels $\{\gamma_{\delta}\}$ defined by \cref{Con_Kernel_d} in 2D, namely
\begin{equation}\label{const_kernel}
     \gamma_{\delta}({\bf x},{\bf y})=
     \left \{
     \begin{array}{cl}
     4/\pi/\delta^4  ,& |{\bf x}-{\bf y}|\leq \delta, \\
     0               ,& |{\bf x}-{\bf y}|> \delta.
     \end{array}
     \right .
\end{equation}
As in \cite{d2021cookbook} the manufactured solution $u_{\delta}({\bf x})=x_{1}^2x_{2}+x_{2}^{2}$
is used to obtain the RHS function $f_{\delta}({\bf x})=-2(x_{2}+1)$ for ${\bf x}\in\Omega$ and VC $g_{\delta}({\bf x}) = u_{\delta}({\bf x})$
for ${\bf x}\in\Omdc$. 
\end{example}

The solution of the local problem \cref{local_diffusion} with a RHS function $f_{0}({\bf x})=-2(x_{2}+1)$ and boundary function $g_{0}({\bf x})=x_{1}^2x_{2}+x_{2}^{2}$ is
$u_{0}({\bf x})=u_{\delta}({\bf x})|_{\Omega}$,
while we take
\begin{equation*}
\widetilde{u}_{0}({\bf x})=x_{1}^2x_{2}+x_{2}^{2}\:\: \mbox{for}\:\: {\bf x}\in\Omdh.
\end{equation*}
In fact in this example $g_{\delta}=\widetilde{u}_{0}$ on $\Omdc$,
$f_{\delta}=f_{0}$ on $\Omega$ holds in \cref{Thm:Nonlocal-local-stand}.
Thus $u_{\delta}=\widetilde{u}_{0}\:\:\mbox{on}\:\:\Omdh$, and the term $\delta^{(3+\mu)/2}$ in \cref{udh_del,udh_L2,Error_Total_Del-sym,Error_Total-sym} vanishes. Then we report on the errors in different norms of $u_{\delta}^{h}$ and $u_{\delta,n_{\delta}}^{h}$
against $\widetilde{u}_{0}$.

\subsubsection{Numerical results for a fixed horizon parameter}\label{Subsec:fix_del}
We fix $\delta=0.4$ as the mesh is refined with a decreasing $h$, and then study the errors
and convergence rates of the \emph{exactcaps} and \emph{nocaps} solutions. For the polygon $\Omega$, the corresponding interaction domain $\Omdc$ has rounded corners, thus it cannot be fully triangulated into elements with straight sides. As pointed out in \cite{d2021cookbook}, $\Omdc$ can be approximated by a polygonal domain with sharp corners to replace the rounded corners while avoiding the extension of the function $g_{\delta}$. The corresponding (consistent) mesh $\mathcal{T}_{\delta}^{h}$ and the solution $u_{\delta}^{h}$ for $h=0.05$ are plotted in \cref{Fig:Exam1:ConMesh-Solution}. Since the figure of $u_{\delta,n_{\delta}}^{h}$ is similar to that of $u_{\delta}^{h}$, the former is omitted. We also use a non-consistent mesh to obtain $u_{\delta}^{h}$, the corresponding mesh and the solution are plotted
in \cref{Fig:Exam1:NonconMesh-Solution}. Since the solutions of two strategies (using consistent and
non-consistent meshes) are rather similar, we take the consistent mesh for later numerical experiments.

\begin{figure}[tbhp]
\centering
\subfloat[the consistent mesh $\mathcal{T}_{0.4}^{0.05}$]{\includegraphics[width=5cm]{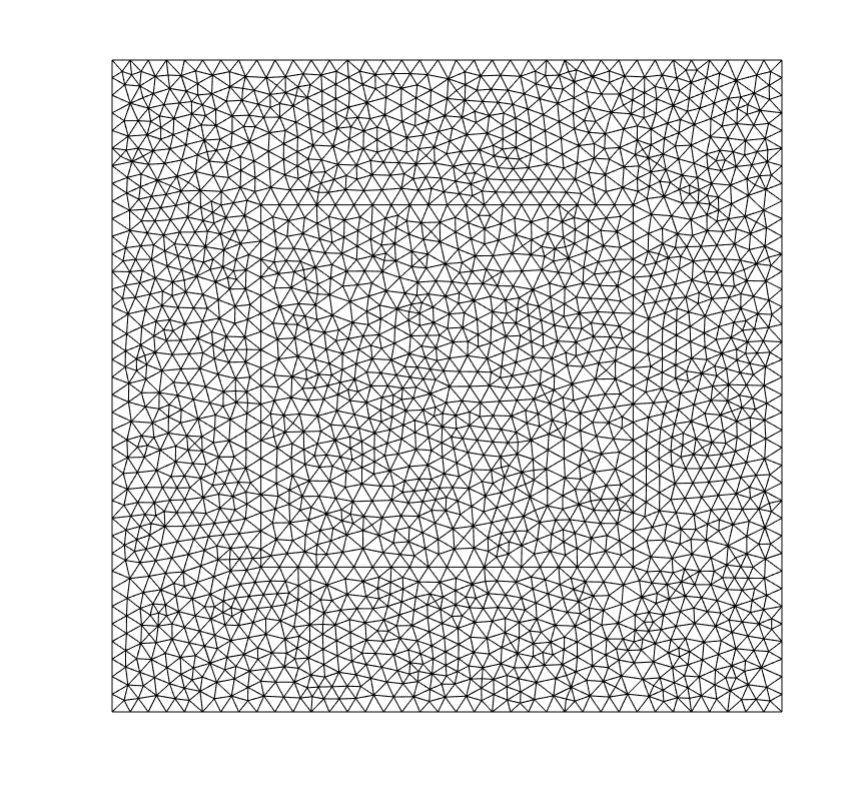}}
\subfloat[the solution $u_{0.4}^{0.05}$]{\includegraphics[width=6cm]{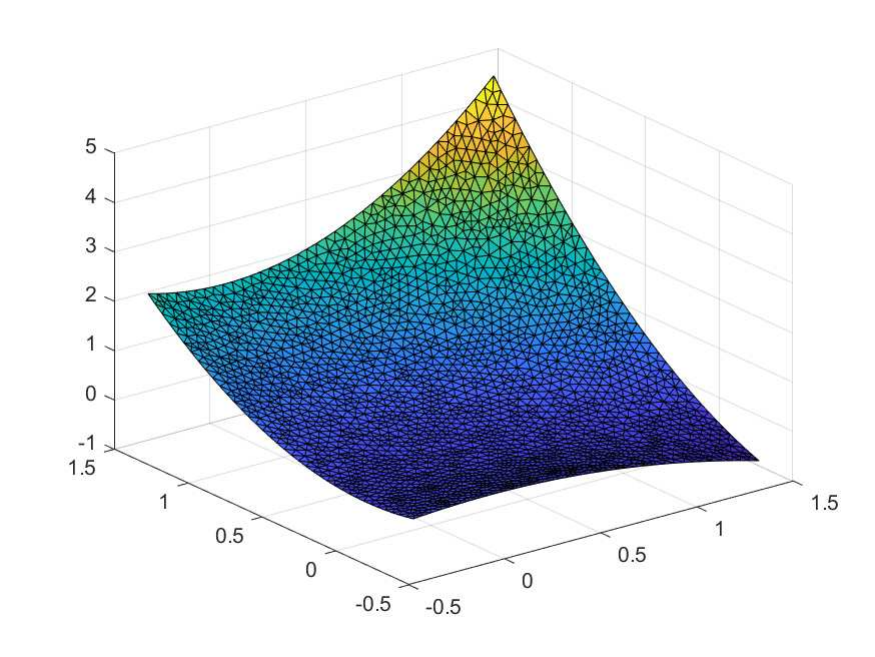}}
\caption{\cref{Example:Continuous}: the consistent mesh and corresponding \emph{exactcaps} solution $u_{\delta}^{h}$, $\delta=0.4$, $h=0.05$}\label{Fig:Exam1:ConMesh-Solution}
\end{figure}

\begin{figure}[tbhp]
\centering
\subfloat[the non-consistent mesh $\mathcal{T}_{0.4}^{h,H}$]{\includegraphics[width=5.5cm]{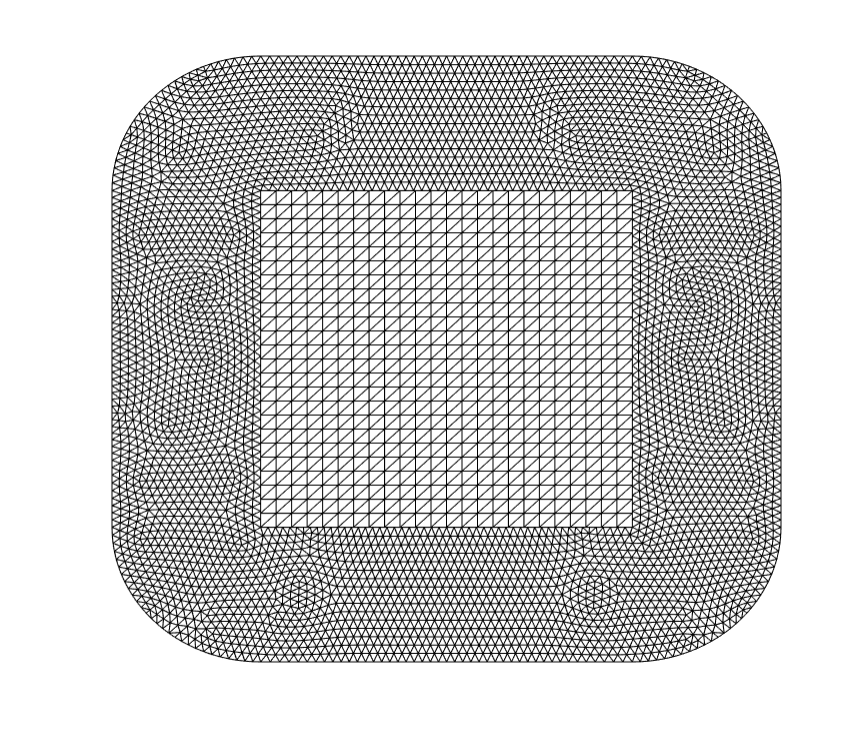}}
\subfloat[the solution $u_{0.4}^{h,H}$]{\includegraphics[width=6cm]{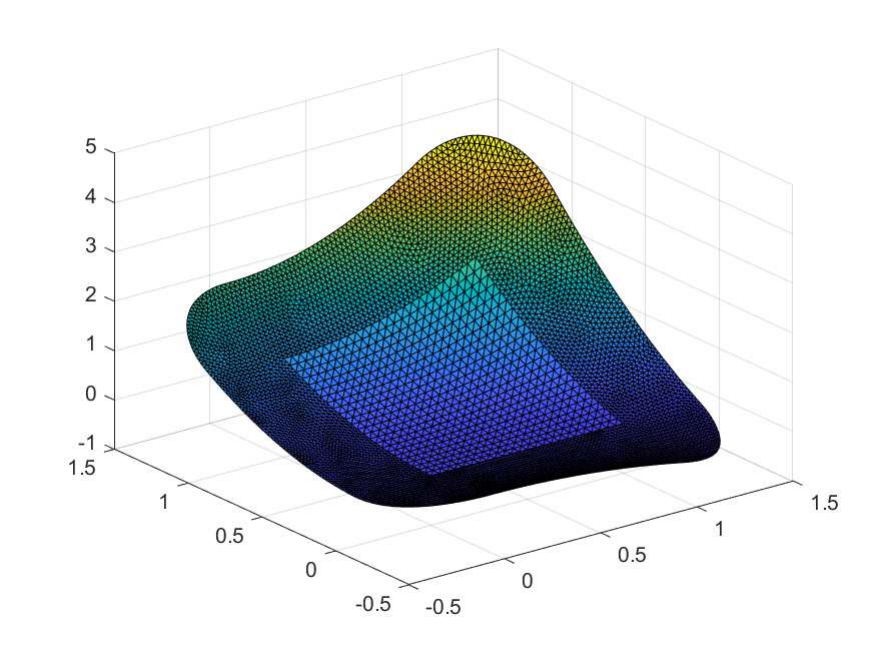}}
\caption{\cref{Example:Continuous}: the non-consistent mesh and corresponding solution $u_{\delta}^{h,H}$, $\delta=0.4$, $h=0.05\sqrt{2}$, $H=3h/8$}\label{Fig:Exam1:NonconMesh-Solution}
\end{figure}

Panel (B) of \cref{Fig:Exam1:ConMesh-Solution} shows that
the solutions $u_{\delta}^{h}$ are continuous across $\partial\Omega$. In fact, in this example, they are expected to be continuous since the analytic expression of $u_{\delta}$ is specified in advance. This assertion could not be applied to $u_{\delta,n_{\delta}}^{h}$ since the analytic expression of $u_{\delta,n_{\delta}}$ is unknown. 
In fact, there is a lack of theory so far to ensure the continuity of $u_{\delta}$ across $\partial\Omega$ in general, although the continuity in $\Omega$ has been discussed in \cite{du2019conforming}. Along this line, the continuity across $\partial\Omega$ for $u_{\delta,n_{\delta}}$ is also not assured. Due to the definition of the linear CDG space, $u_{\delta}^{h}$ and $u_{\delta,n_{\delta}}^{h}$ are allowed to be discontinuous across $\partial\Omega$. Here $u_{0.4}^{0.05}$ might appear to be continuous simply because the jumps across $\partial\Omega$ are smaller than the scales of the solution itself. If we zoom in to examine these solutions in greater detail, especially for larger $h$, the discontinuity can become more visible. For this purpose, we  present zoomed plots around point $(1,1)$ of the solutions $u_{\delta}^{h}$ and $u_{\delta,n_{\delta}}^{h}$ for $h=\delta$ in \cref{Fig:Exam1:Solutionstep1_zoom}.

\begin{figure}[tbhp]
\centering
\subfloat[the solution $u_{0.4}^{0.4}$]{\includegraphics[width=6cm]{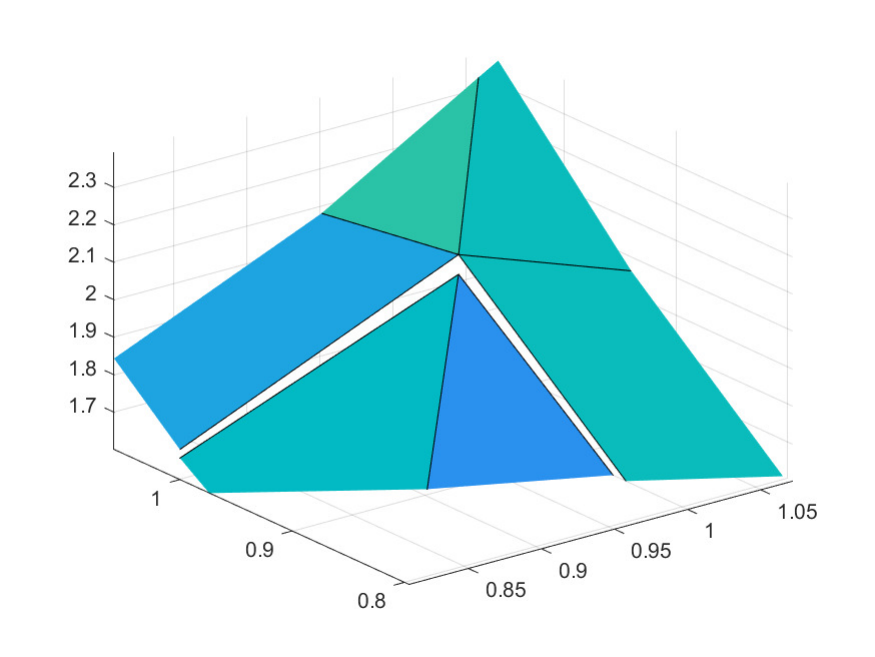}}
\subfloat[the solution $u_{0.4,n_{0.4}}^{0.4}$]{\includegraphics[width=6cm]{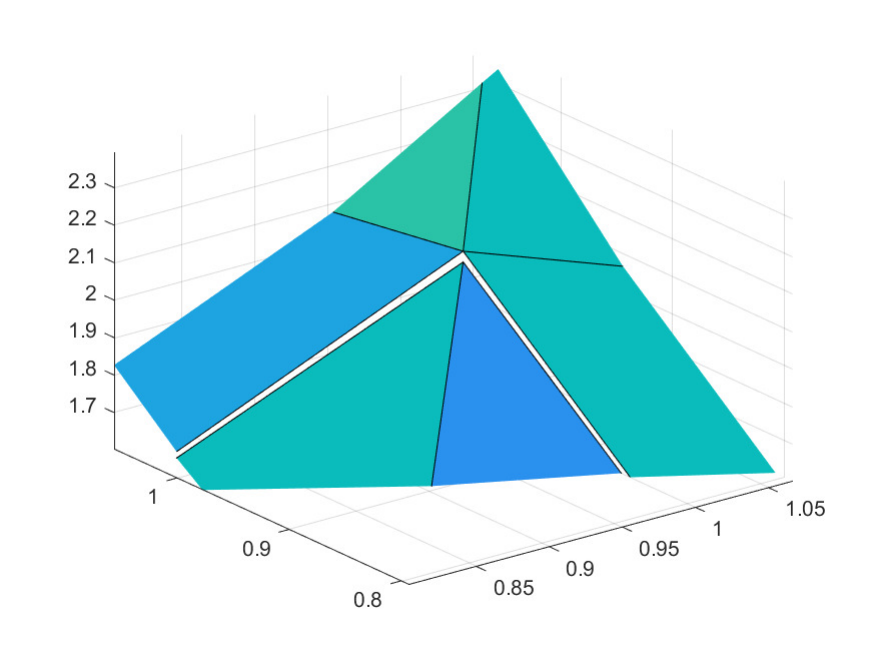}}
\caption{\cref{Example:Continuous}: zoom around point $(1,1)$ of $u_{\delta}^{h}$ and
$u_{\delta,n_{\delta}}^{h}$, $h=\delta=0.4$}\label{Fig:Exam1:Solutionstep1_zoom}
\end{figure}

It is seen from \cref{table:Exam1:hconv} that the
solution $u_{\delta}^{h}$ has smaller errors in the energy norm than that of $u_{\delta,n_{\delta}}^{h}$, while the opposite behavior is observed in the case of the $L^{2}$ norm. However, we do observe the second-order convergence rates for both methods in both types of norms. This confirms the theoretical results since $\delta$ is taken to be a constant in this set of experiments, that is
by \cref{udh_del,udh_L2,Error_Total_Del-sym,Error_Total-sym} without the term $\delta^{(3+\mu)/2}$, 
the errors in the energy and $L^{2}$ norms of $u_{\delta}^{h}$ and $u_{\delta,n_{\delta}}^{h}$ are of $\delta^{-1}h^{2}$ and $\delta^{-3}h^{2}$ order, respectively.

\begin{table}[ht]
\setlength{\tabcolsep}{4pt}
\centering
\caption{\cref{Example:Continuous}: errors and convergence rate of nonlocal numerical solutions $u_{\delta}^{h}$ and $u_{\delta,n_{\delta}}^{h}$ against the local exact solution $\widetilde{u}_{0}$ in the energy norms and $L^2$ norms, $\delta=0.4$}\label{table:Exam1:hconv} \footnotesize
\begin{tabular}{|c|ccccc|cccc|}
\hline
& \multicolumn{5}{c|}{Energy norm}& \multicolumn{4}{c|}{$L^{2}$ norm}\\
\hline
$\frac{\delta}{h}$ & $\|\widetilde{u}_{0}$-$u_{\delta}^{h}\!\|_{\delta}$ &  Rate & $\|\widetilde{u}_{0}$-$u_{\delta,n_{\delta}}^{h}\|_{\delta,n_{\delta}}$
& Rate & $\|\widetilde{u}_{0}$-$u_{\delta,n_{\delta}}^{h}\!\|_{\delta}$ & $\|u_{0}$-$u_{\delta}^{h}\|$ & Rate & $\|u_{0}$-$u_{\delta,n_{\delta}}^{h}\!\|$&
Rate \\
\hline
$2^{0}$ &6.38e-2& -- &1.22e-1& -- &1.25e-1&1.58e-2& -- &1.40e-2& --  \\                                                        
\hline                                                         
$2^{1}$ &1.01e-2&2.66&2.51e-2&2.28&2.53e-2&6.11e-3&1.37&2.08e-3&2.75 \\
\hline                                                           
$2^{2}$ &2.83e-3&1.84&6.67e-3&1.91&6.69e-3&1.46e-3&2.06&4.65e-4&2.16 \\
\hline                                                           
$2^{3}$ &6.49e-4&2.12&1.58e-3&2.07&1.58e-3&3.70e-4&1.98&1.27e-4&1.88 \\
\hline                                                           
$2^{4}$ &1.58e-4&2.04&3.95e-4&2.00&3.95e-4&9.14e-5&2.02&3.17e-5&2.00 \\
\hline                                                           
$2^{5}$ &3.96e-5&1.99&9.76e-5&2.02&9.76e-5&2.28e-5&2.01&8.11e-6&1.97 \\
\hline
\end{tabular}
\end{table}

\begin{table}[ht]
\centering
\caption{\cref{Example:Continuous}: errors $\|\widetilde{u}_{0}-u_{\delta}^{h}\|_{\delta}$ and convergence rates, $\delta=mh$}\label{table:Exam1:ac:exactcap:energynorm} \footnotesize
\begin{tabular}{|c|cccccc|}
\hline $\delta_{0}/\delta$ & $2^{0}$ & $2^{1}$ & $2^{2}$ & $2^{3}$ & $2^{4}$& $2^{5}$\\
\hline
$\delta_{0}$=0.4, $m$=2 & 1.01e-2 & 5.12e-3 & 2.43e-3 & 1.16e-3 & 5.73e-4 & 2.84e-4 \\
 Rate                   & --      & 0.98    & 1.08    & 1.07    & 1.02    & 1.01    \\
\hline
$\delta_{0}$=0.6, $m$=3 & 6.89e-3 & 3.60e-3 & 1.72e-3 & 8.14e-4 & 4.02e-4 & 1.98e-4 \\
 Rate                   & --      & 0.94    & 1.07    & 1.08    & 1.02    & 1.02    \\
\hline
$\delta_{0}$=0.8, $m$=4 & 5.13e-3 & 2.83e-3 & 1.30e-3 & 6.19e-4 & 3.08e-4 & 1.51e-4 \\
 Rate                   & --      & 0.86    & 1.12    & 1.07    & 1.01    & 1.03    \\
\hline
$\delta_{0}$=1.0, $m$=5 & 3.90e-3 & 2.26e-3 & 1.02e-3 & 4.99e-4 & 2.48e-4 & 1.21e-4 \\
 Rate                   & --      & 0.79    & 1.15    & 1.03    & 1.01    & 1.04    \\
\hline
\end{tabular}
\end{table}

\begin{table}[ht]
\centering
\caption{\cref{Example:Continuous}: errors $\|\widetilde{u}_{0}-u_{\delta}^{h}\|_{\delta}$, $\|\widetilde{u}_{0}-u_{\delta,n_{\delta}}^{h}\|_{\delta,n_{\delta}}$ and convergence rates, $h$ fixed}\label{table:Exam1:ac:exactcap:energynorm:hfixed} \footnotesize
\begin{tabular}{|c|cccc|cccc|}
\hline
& \multicolumn{4}{c|}{$\|\widetilde{u}_{0}-u_{\delta}^{h}\|_{\delta}$}& \multicolumn{4}{c|}{$\|\widetilde{u}_{0}-u_{\delta,n_{\delta}}^{h}\|_{\delta,n_{\delta}}$}\\
\hline 
$h$     & $\delta=2h$ & $\delta=3h$ & $\delta=4h$ & $\delta=5h$ 
        & $\delta=2h$ & $\delta=3h$ & $\delta=4h$ & $\delta=5h$ \\ 
\hline
0.2     & 1.01e-2 & 6.89e-3 & 5.13e-3 & 3.90e-3& 2.51e-2 & 1.37e-2 & 9.35e-3 & 7.25e-3 \\
Rate    & --      & -0.94   & -1.03   &  -1.23 & --      & -1.49   & -1.33   & -1.14   \\  
\hline                                                                                
0.1     & 5.12e-3 & 3.60e-3 & 2.83e-3 & 2.26e-3& 2.73e-2 & 1.02e-2 & 6.67e-3 & 4.81e-3 \\
Rate    & --      & -0.87   & -0.84   & -1.01  & --      & -2.43   & -1.48   & -1.47   \\
\hline                                                                                
0.05    & 2.43e-3 & 1.72e-3 & 1.30e-3 & 1.02e-3& 3.55e-2 & 1.21e-2 & 5.91e-3 & 3.83e-3 \\
Rate    & --      & -0.85   & -0.97   & -1.09  & --      & -2.65   & -2.49   & -1.94   \\
\hline                                                                                
0.025   & 1.16e-3 & 8.14e-4 & 6.19e-4 & 4.99e-4& 6.51e-2 & 1.77e-2 & 7.81e-3 & 4.41e-3 \\
Rate    & --      & -0.87   & -0.95   & -0.97  & --      & -3.21   & -2.84   & -2.56   \\
\hline                                                                                
0.0125  & 5.73e-4 & 4.02e-4 & 3.08e-4 & 2.48e-4& 1.36e-1 & 3.66e-2 & 1.48e-2 & 7.54e-3 \\
Rate    & --      &-0.87    & -0.93   & -0.97  & --      & -3.24   & -3.15   & -3.02   \\
\hline                                                                                
0.00625 & 2.84e-4 & 1.98e-4 & 1.51e-4 & 1.21e-4& 2.67e-1 & 7.06e-2 & 2.94e-2 & 1.50e-2 \\
Rate    & --      & -0.89   & -0.94   & -0.99  & --      & -3.28   & -3.05   & -3.02   \\
\hline
\end{tabular}
\end{table}

\subsubsection{Numerical results  for the case of a fixed ratio between the horizon parameter and the mesh size}\label{Subsec:fix_r}
In this part, we fix $m=\delta/h$ as a constant while refining $\delta$. In this case \cref{udh_del}
and \cref{udh_L2} without the term $\delta^{(3+\mu)/2}$ turn out as
\begin{align}
\left\|\widetilde{u}_{0}-u_{\delta}^{h}\right\|_{\delta}\lesssim
\delta m^{-1},\label{udh_const_r-d}\\
\left\|u_{0}-u_{\delta}^{h}\right\|_{L^{2}\left(\Omega\right)}\lesssim
\delta m^{-1}.\label{udh_const_r-L2}
\end{align}

\cref{table:Exam1:ac:exactcap:energynorm} provides errors and
convergence rates of \emph{exactcaps} solution $u_{\delta}^{h}$ against the local exact solution $u_{0}$ in the energy norm with
$m=2$, $3$, $4$, and $5$, respectively. The convergence rates all exhibit the first order with respect to $\delta$, which are consistent with \cref{udh_const_r-d}. We also use the same data of errors to calculate the convergence rates with respect to $\delta$ for fixed $h$. To this end, notice that $h$ is constant in each column of \cref{table:Exam1:ac:exactcap:energynorm}. Thus, we rotate \cref{table:Exam1:ac:exactcap:energynorm} $90$ degrees to get the middle part of \cref{table:Exam1:ac:exactcap:energynorm:hfixed}. It shows nearly $-1$ order with respect to $\delta$ for a fixed $h$, which coincides with the error estimate \cref{udh_del} without the term $\delta^{(3+\mu)/2}$. 

\begin{table}[ht]
\centering
\caption{\cref{Example:Continuous}: errors $\|u_{0}-u_{\delta}^{h}\|_{L^{2}(\Omega)}$ and convergence rates, $\delta=mh$}\label{table:Exam1:ac:exactcap:L2norm} \footnotesize
\begin{tabular}{|c|cccccc|}
\hline $\delta_{0}/\delta$& $2^{0}$ & $2^{1}$ & $2^{2}$ & $2^{3}$ & $2^{4}$& $2^{5}$\\
\hline
$\delta_{0}$=0.4, $m$=2 &6.11e-3 & 1.56e-3 & 3.61e-4 & 9.75e-5 & 2.32e-5 & 4.92e-6 \\
 Rate                   & --     & 1.97    & 2.11    & 1.89    & 2.07    & 2.24    \\
\hline
$\delta_{0}$=0.6, $m$=3 &5.97e-3 & 1.75e-3 & 4.47e-4 & 1.08e-4 & 2.72e-5 & 7.03e-6 \\
 Rate                   & --     & 1.77    & 1.97    & 2.05    & 1.99    & 1.95    \\
\hline
$\delta_{0}$=0.8, $m$=4 &5.97e-3 & 1.46e-3 & 3.55e-4 & 1.01e-4 & 2.38e-5 & 5.68e-6 \\
 Rate                   & --     & 2.03    & 2.04    & 1.81    & 2.09    & 2.07    \\
\hline
$\delta_{0}$=1.0, $m$=5 &6.31e-3 & 1.49e-3 & 3.84e-4 & 9.75e-5 & 2.44e-5 & 6.07e-6 \\
 Rate                   & --     & 2.08    & 1.96    & 1.98    & 2.00    & 2.01    \\ 
\hline
\end{tabular}
\end{table}

In \cref{table:Exam1:ac:exactcap:L2norm} the results in the $L^{2}$ norm are reported. The convergence rates show the second order, across the table, which is better than that predicted by \cref{udh_const_r-L2}. Furthermore, we hardly observe any obvious drop or growth of errors when $\delta$ decreases for fixed $h$, see the data along each column. The two findings indicate that the error estimate \cref{udh_L2} is likely not sharp. It seems \cref{udh_L2} may be improved as
\begin{equation}\label{udh_L2_improve}
\left\|u_{0}-u_{\delta}^{h}\right\|_{L^{2}\left(\Omega\right)}\lesssim \delta^{2}+h^{2},
\end{equation}
which suggests the possible validity of the Aubin-Nitsche technique with respect to the horizon parameter, see also \cref{Nitche}.

Since $\delta=mh$, \cref{Error_Total_Del-sym} and \cref{Error_Total-sym} turn out as
\begin{align}
\|\widetilde{u}_{0}-u_{\delta,n_{\delta}}^{h}\|_{\delta,n_{\delta}}\lesssim \delta^{2}+
\delta^{-3}h^{2}\approx \delta^{-1} m^{-2},\label{udnd_const_r-d-nsym}\\
\|u_{0}-u_{\delta,n_{\delta}}^{h}\|_{L^{2}\left(\Omega\right)}\lesssim \delta^{2}+
\delta^{-3}h^{2}\approx \delta^{-1}m^{-2}.\label{udnd_const_r-L2-nsym}
\end{align}

\begin{table}[ht]
\centering
\caption{\cref{Example:Continuous}: errors $\|\widetilde{u}_{0}-u_{\delta,n_{\delta}}^{h}\|_{\delta,n_{\delta}}$ and convergence rates,  $\delta=mh$}\label{table:Exam1:ac:nocap:energynorm} \footnotesize
\begin{tabular}{|c|cccccc|}
\hline $\delta_{0}/\delta$ & $2^{0}$ & $2^{1}$ & $2^{2}$ & $2^{3}$ & $2^{4}$ & $2^{5}$\\
\hline
$\delta_{0}$=0.4, $m$=2 & 2.51e-2 & 2.73e-2  & 3.55e-2  & 6.51e-2  & 1.36e-1  & 2.67e-1 \\
 Rate & --      & -0.12    & -0.38    & -0.87    & -1.06    & -0.97   \\
\hline
$\delta_{0}$=0.6, $m$=3 & 1.37e-2 & 1.02e-2  & 1.21e-2  & 1.77e-2  & 3.66e-2  & 7.06e-2 \\
 Rate & --      & 0.43     & -0.25    & -0.55    & -1.05    & -0.95   \\
\hline
$\delta_{0}$=0.8, $m$=4 & 9.35e-3 & 6.67e-3  & 5.91e-3  & 7.81e-3  & 1.48e-2  & 2.94e-2 \\
 Rate & --      & 0.49     & 0.17     & -0.40    & -0.92    & -0.99   \\
\hline
$\delta_{0}$=1.0, $m$=5 & 7.25e-3 & 4.81e-3  & 3.83e-3  & 4.41e-3  & 7.54e-3  & 1.50e-2 \\
 Rate & --      & 0.59     & 0.33     & -0.20    & -0.77    & -0.99   \\
\hline
\end{tabular}
\end{table}

\cref{table:Exam1:ac:nocap:energynorm} provides errors and
convergence rates of \emph{nocaps} solution $u_{\delta,n_{\delta}}^{h}$ against the local exact solution $u_{0}$ in the energy norm. The convergence rates are near $-1$ order with respect to $\delta$ for decreasing $\delta$ under the setting $\delta=mh$ (as seen for each row), which coincides with \cref{udnd_const_r-d-nsym}. We also calculate the convergence rates with respect to $\delta$ for fixed $h$, see the right part of \cref{table:Exam1:ac:exactcap:energynorm:hfixed}. They show nearly $-3$ order, which coincides with \cref{Error_Total_Del-sym} without the term $\delta^{(3+\mu)/2}$.
In \cref{table:Exam1:ac:nocap:L2norm} the results
in the $L^{2}$ norm are reported. On the one hand, errors stabilize gradually around a certain value when $\delta$ decreases under the setting $\delta=mh$, which is better than that predicted by \cref{udnd_const_r-L2-nsym}. On the other hand, the rates show nearly $-2$ order with respect to $\delta$ for fixed $h$ (if computed along each column, similar to \cref{table:Exam1:ac:exactcap:energynorm:hfixed}, the numerical values are omitted to save space). The two findings indicate that the error estimate \cref{Error_Total-sym} is not sharp, and a possible improvement of \cref{Error_Total-sym} may be given by
\begin{equation}\label{Error_Total-sym-improve}
\|u_{0}-u_{\delta,n_{\delta}}^{h}\|_{L^{2}\left(\Omega\right)}\lesssim
\delta^{-2}h^{2}+\delta^{2}+h^{2}.
\end{equation}

\begin{table}[ht]
\centering
\caption{\cref{Example:Continuous}: errors $\|u_{0}-u_{\delta,n_{\delta}}^{h}\|_{L^{2}(\Omega)}$ and convergence rates, $\delta=mh$}
\label{table:Exam1:ac:nocap:L2norm} \footnotesize
\begin{tabular}{|c|cccccc|}
\hline $\delta_{0}/\delta$  & $2^{0}$ & $2^{1}$ & $2^{2}$ & $2^{3}$ & $2^{4}$ & $2^{5}$\\
\hline
$\delta_{0}$=0.4, $m$=2  & 2.08e-3  & 3.06e-3  & 3.52e-3  & 2.70e-3  & 2.20e-3 & 2.49e-3 \\
 Rate  & --       & -0.56    & -0.20    & 0.38     & 0.30    & -0.18   \\
\hline
$\delta_{0}$=0.6, $m$=3  & 2.68e-3  & 8.36e-4  & 1.66e-3  & 1.44e-3  & 1.21e-3 & 1.08e-3 \\
 Rate  & --       & 1.68     & -0.99    & 0.21     & 0.25    & 0.16    \\
\hline
$\delta_{0}$=0.8, $m$=4  & 3.39e-3  & 4.65e-4  & 7.63e-4  & 8.21e-4  & 7.33e-4 & 7.20e-4 \\
 Rate  & --       & 2.87     & -0.71    & -0.11    & 0.16    & 0.03    \\
\hline
$\delta_{0}$=1.0, $m$=5  & 3.76e-3  & 4.95e-4  & 4.64e-4  & 5.45e-4  & 4.71e-4 & 4.63e-4 \\
 Rate  & --       & 2.93     & 0.09     & -0.23    & 0.21    & 0.02    \\
\hline
\end{tabular}
\end{table}

\subsubsection{Numerical results for power law between the horizon parameter and
the mesh size}
The sharpness of the estimates \cref{udh_del} and \cref{Error_Total_Del-sym} has been discussed to some extent in the two subsections above. Here we further verify this sharpness by setting $h=O\left(\delta^{\beta}\right)$. By \cref{udh_del} and \cref{udh_L2} without the term $\delta^{(3+\mu)/2}$, it is expected that
\begin{equation}\label{Error_Numer_ex1}
\left\|\widetilde{u}_{0}-u_{\delta}^{h}\right\|_{\delta}\lesssim
\delta^{-1}h^{2}\sim\delta^{2\beta-1},\:\:
\left\|u_{0}-u_{\delta}^{h}\right\|_{L^{2}\left(\Omega\right)}\lesssim
\delta^{-1}h^{2}\sim\delta^{2\beta-1}.
\end{equation}
The errors for $\beta=1.1, \cdots, 2.0$ are plotted in \cref{Fig:Exam1:betagt1:exactcap}. Here $\delta$ is reduced to two-thirds of the previous step each time. Due to the increasing demand on the CPU time  
as $\beta$ increases, we take fewer $\delta$ refinement for larger $\beta$. In (A), the errors versus $\delta$ in the energy norm are plotted. We find that the convergence rates are of $2\beta-1$ order which is consistent with the estimate in the energy norm in \cref{Error_Numer_ex1}. In (B) the errors in the $L^2$ norm show the $2\beta$ order, which indicates again \cref{udh_L2} may be improved as \cref{udh_L2_improve}.

\begin{figure}[tbhp]
\centering
\subfloat[the energy norm]{\includegraphics[width=6cm]{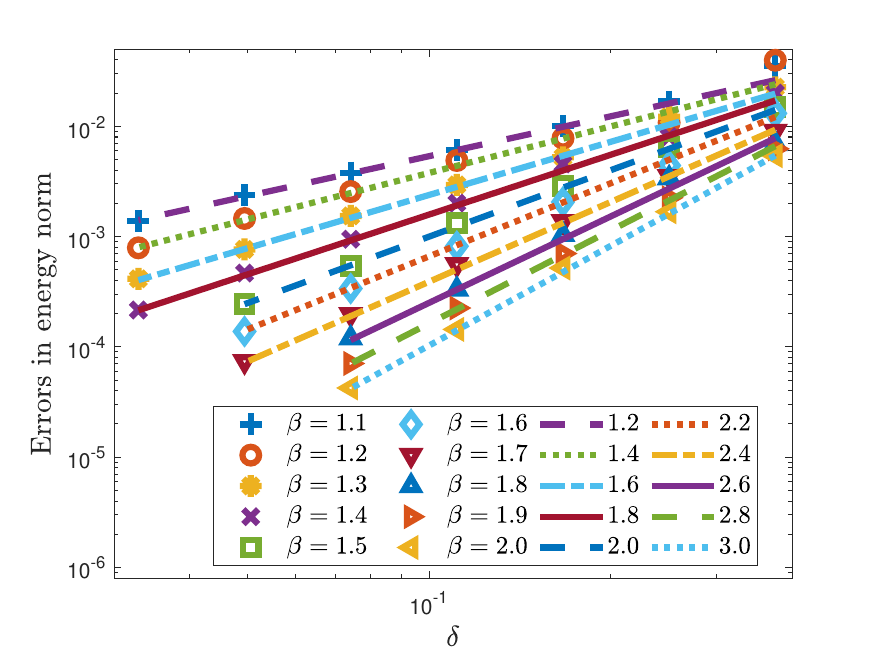}}
\subfloat[the $L^2$ norm]{\includegraphics[width=6cm]
{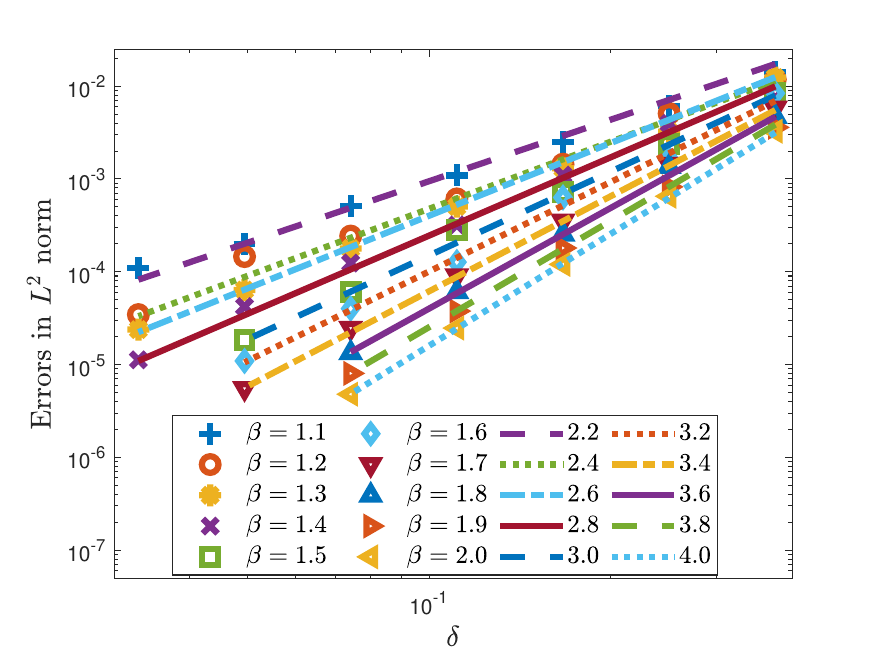}}
\caption{\cref{Example:Continuous}: errors and convergence
rates for \emph{exactcaps} solution $u_{\delta}^{h}$ against the local exact solution $u_{0}$ under the setting $h=O\left(\delta^{\beta}\right)$, $\beta\in(1,2]$}\label{Fig:Exam1:betagt1:exactcap}
\end{figure}

\begin{figure}[tbhp]
\centering
\subfloat[the energy norm]{\includegraphics[width=6cm]{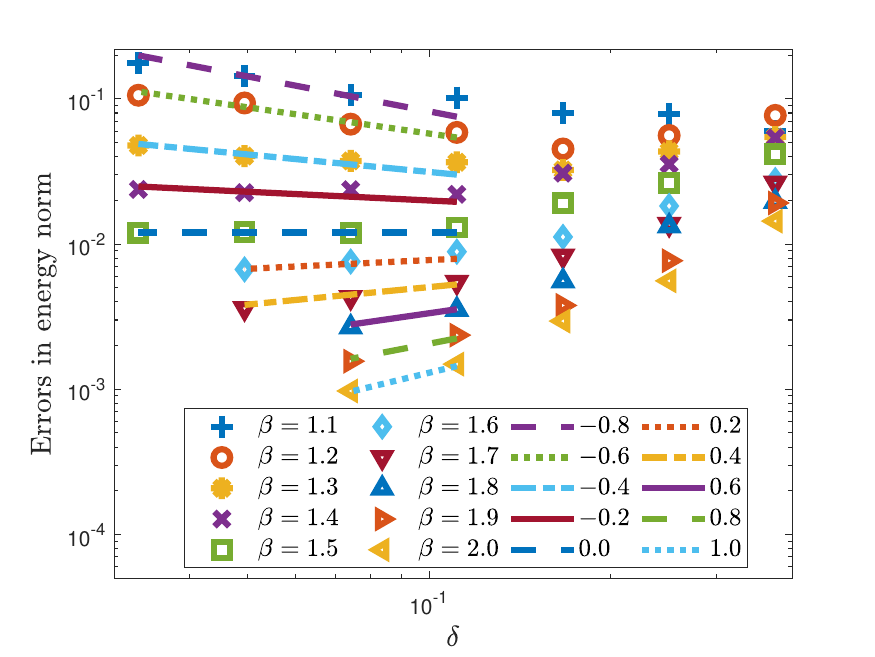}}
\subfloat[the $L^2$ norm]{\includegraphics[width=6cm]
{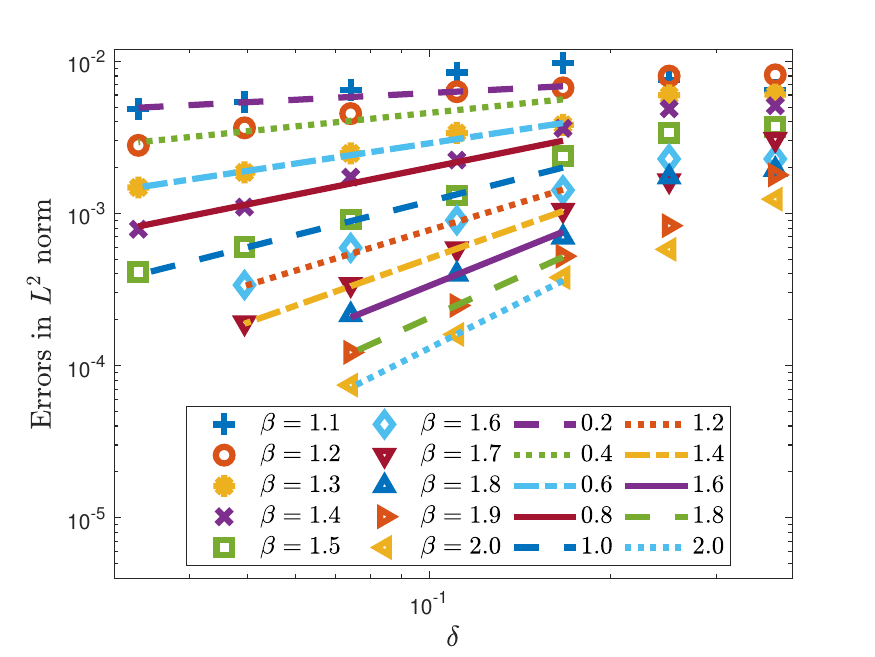}}
\caption{\cref{Example:Continuous}: errors in different norms
for \emph{nocaps} solution $u_{\delta,n_{\delta}}^{h}$ against the local exact solution $u_{0}$ under the setting $h=O\left(\delta^{\beta}\right)$, $\beta\in(1,2]$}\label{Fig:Exam1:CondAC}
\end{figure}

Since $h=O\left(\delta^{\beta}\right)$, \cref{Error_Total_Del-sym}
and \cref{Error_Total-sym} turn out to be given by
\begin{equation}\label{cond_AC_sym}
\|\widetilde{u}_{0}-u_{\delta,n_{\delta}}^{h}\|_{\delta,n_{\delta}}\lesssim
\delta^{2\beta-3},\:\:
\|u_{0}-u_{\delta,n_{\delta}}^{h}\|_{L^{2}\left(\Omega\right)}\lesssim
\delta^{2\beta-3}.
\end{equation}
We plot the results of $u_{\delta,n_{\delta}}^{h}$ in \cref{Fig:Exam1:CondAC}. The errors versus $\delta$ in the energy norm are plotted in (A). The convergence rates
show nearly $2\beta-3$ order which coincides with the estimate in the energy norm in \cref{cond_AC_sym}. It is worth mentioning that errors in the energy norm decrease monotonically only for $\beta>1.5$. In (B) the errors in the $L^2$ norm show nearly $2\beta-2$ order which again suggests \cref{Error_Total-sym} be improved as
\cref{Error_Total-sym-improve}.

\subsection{Nonlocal problems with a perturbed RHS function and VC}\label{subsec:Numer:pert}
\begin{example}\label{Example:Discontinuous}
We consider nonlocal problems \cref{nonlocal_diffusion} on the domain $\Omega=(0,1)^{2}$
with the family of kernels $\{\gamma_{\delta}\}$ defined by \cref{const_kernel}. We set the RHS function as
\begin{equation}\label{rhs_pert}
f_{\delta}({\bf x})=-2(x_{2}+1)+\delta^2e^{x_{1}^2+3x_{2}^2}, \:\:\mbox{for}\:\:
{\bf x}\in\Omega,
\end{equation}
and the following two kinds of VCs
\begin{align}
g_{\delta}({\bf x}) &= x_{1}^2x_{2}+x_{2}^{2}+\delta^2\sin(x_{1}-2x_{2}),\:\:\mbox{for}\:\: {\bf x}\in\Omdc,\label{dsquare}\\
g_{\delta}({\bf x}) &= x_{1}^2x_{2}+x_{2}^{2}+\delta^3\sin(x_{1}-2x_{2}),\:\:\mbox{for}\:\: {\bf x}\in\Omdc,\label{dtriple}
\end{align}
which satisfy \cref{bc-rhs-est} with $\mu=0$ and $\mu=1$, respectively. The corresponding local problem is the same as \cref{Example:Continuous}. Here the nonlocal solutions $u_{\delta}$ and $u_{\delta,n_{\delta}}$ are all discontinuous across $\partial\Omega$.
\end{example}

Most numerical results and findings here are similar to those in \cref{subsec:Numer:exact}, so we do not repeat such discussions and only show the numerical results for the case of a fixed ratio between the horizon parameter and the mesh size. \cref{table:Exam2:d2:mcons2,table:Exam2:d3:mcons2} provide the results of \emph{exactcaps} and \emph{nocaps} solutions with $m=2$, for VCs \cref{dsquare,dtriple}, respectively. The convergence rates of $u_{\delta}^{h}$ for both VCs against $\widetilde{u}_{0}$ in the $L^{2}$ norm show second order, which is similar to the case $m=2$ in \cref{table:Exam1:ac:exactcap:L2norm}. 

\begin{table}[ht]
\centering
\caption{\cref{Example:Discontinuous}: errors and convergence rate of nonlocal numerical solutions $u_{\delta}^{h}$ and $u_{\delta,n_{\delta}}^{h}$ with VC \cref{dsquare} against the local exact solution $\widetilde{u}_{0}$, $\delta_{0}=0.4$, $\delta=2h$}\label{table:Exam2:d2:mcons2} \footnotesize
\begin{tabular}{|c|cccc|cccc|}
\hline
& \multicolumn{4}{c|}{Energy norm}& \multicolumn{4}{c|}{$L^{2}$ norm}\\
\hline
$\delta_{0}/\delta$ & $\|\widetilde{u}_{0}-u_{\delta}^{h}\|_{\delta}$ &  Rate & $\|\widetilde{u}_{0}$-$u_{\delta,n_{\delta}}^{h}\|_{\delta,n_{\delta}}$ &  Rate
 & $\|u_{0}-u_{\delta}^{h}\|$ &  Rate & $\|u_{0}$-$u_{\delta,n_{\delta}}^{h}\|$ & Rate\\
\hline
$2^{0}$ & 2.80e-1 &  --  & 2.74e-1 &  --   & 5.83e-2 &  --   & 5.61e-2 &  --    \\
\hline  
$2^{1}$ & 1.25e-1 & 1.17 & 1.29e-1 & 1.09  & 1.49e-2 & 1.97  & 1.54e-2 & 1.86   \\
\hline  
$2^{2}$ & 5.15e-2 & 1.28 & 6.46e-2 & 0.99  & 4.08e-3 & 1.87  & 5.74e-3 & 1.42   \\
\hline  
$2^{3}$ & 1.96e-2 & 1.40 & 6.93e-2 & -1.01 & 1.06e-3 & 1.94  & 2.94e-3 & 0.97   \\
\hline  
$2^{4}$ & 7.20e-3 & 1.44 & 1.36e-1 & -0.97 & 2.73e-4 & 1.96  & 2.19e-3 & 0.42   \\
\hline  
$2^{5}$ & 2.60e-3 & 1.47 & 2.67e-1 & -0.97 & 6.92e-5 & 1.98  & 2.46e-3 & -0.16  \\
\hline
\end{tabular}
\end{table}
 
\begin{table}[ht]
\centering
\caption{\cref{Example:Discontinuous}: errors and convergence rate of nonlocal numerical solutions $u_{\delta}^{h}$ and $u_{\delta,n_{\delta}}^{h}$ with VC \cref{dtriple} against the local exact solution $\widetilde{u}_{0}$, $\delta_{0}=0.4$, $\delta=2h$}\label{table:Exam2:d3:mcons2}\footnotesize
\begin{tabular}{|c|cccc|cccc|}
\hline
& \multicolumn{4}{c|}{Energy norm}& \multicolumn{4}{c|}{$L^{2}$ norm}\\
\hline
$\delta_{0}/\delta$ & $\left\|\widetilde{u}_{0}-u_{\delta}^{h}\right\|_{\delta}$ &  Rate & $\|\widetilde{u}_{0}$-$u_{\delta,n_{\delta}}^{h}\|_{\delta,n_{\delta}}$ &  Rate
 & $\left\|u_{0}-u_{\delta}^{h}\right\|$ &  Rate & $\|u_{0}$-$u_{\delta,n_{\delta}}^{h}\|$ &  Rate\\
\hline
$2^{0}$ & 2.28e-1 &  --  & 2.15e-1 &  --   & 6.13e-2 &  --  & 5.75e-2 &  --   \\
\hline
$2^{1}$ & 4.56e-2 & 2.32 & 3.86e-2 & 2.48  & 1.10e-2 & 2.48 & 7.46e-3 & 2.95  \\
\hline
$2^{2}$ & 1.06e-2 & 2.10 & 3.31e-2 & 2.23  & 2.40e-3 & 2.19 & 1.84e-3 & 2.02  \\
\hline
$2^{3}$ & 2.86e-3 & 1.90 & 6.50e-2 & -0.98 & 5.83e-4 & 2.04 & 2.36e-3 & -0.36 \\
\hline
$2^{4}$ & 8.65e-4 & 1.72 & 1.36e-1 & -1.06 & 1.42e-4 & 2.03 & 2.15e-3 & 0.13  \\
\hline
$2^{5}$ & 3.26e-4 & 1.41 & 2.67e-1 & -0.98 & 3.44e-5 & 2.05 & 2.51e-3 & -0.22 \\
\hline
\end{tabular}
\end{table}

The convergence results of $u_{\delta,n_{\delta}}^{h}$ for both VCs against $\widetilde{u}_{0}$ in the energy and the $L^2$ norms exhibit about $-1$ order and zeroth order, respectively. These results are similar to the corresponding results in \cref{table:Exam1:ac:nocap:energynorm,table:Exam1:ac:nocap:L2norm} although the rates in \cref{table:Exam2:d2:mcons2,table:Exam2:d3:mcons2} have much larger oscillations. The reason is similar to the seemingly abnormal rates for $u_{\delta}^{h}$ against $\widetilde{u}_{0}$ in the energy norm, which can be explained as follows.

For errors of $u_{\delta}^{h}$ in the energy norm, \cref{udh_del} turns out to be
\begin{align}
\left\|\widetilde{u}_{0}-u_{\delta}^{h}\right\|_{\delta}&\leq C_{7}\delta^{3/2}+C_{8}\delta
\label{udh_const_r-d_pert}\\
\left\|\widetilde{u}_{0}-u_{\delta}^{h}\right\|_{\delta}&\leq C_{7}\delta^{2}+C_{8}\delta
\label{udh_const_r-d_pert-3}
\end{align}
for VCs \cref{dsquare} and \cref{dtriple}, respectively.
However, the convergence rates in \cref{table:Exam2:d2:mcons2} appear to be $3/2$
order which seems inconsistent with \cref{udh_const_r-d_pert}.  
The rates in \cref{table:Exam2:d3:mcons2} have a decreasing trend. They might eventually reach the first 
order, which would coincide with \cref{udh_const_r-d_pert-3}. In fact, by \cref{Total_Err_En} it holds that
\begin{equation*}
\left\|\widetilde{u}_{0}-u_{\delta}^{h}\right\|_{\delta}\leq
\left\|u_{\delta}-\widetilde{u}_{0}\right\|_{\delta}+
\left\|u_{\delta}-u_{\delta}^{h}\right\|_{\delta}=E_{1}+E_{2}.
\end{equation*}

Recall that in \cref{Example:Continuous} the RHS function and VC are all exact such that $u_{\delta}({\bf x})=\widetilde{u}_{0}({\bf x})$ for ${\bf x}\in\Omdh$. So, the term $E_{1}$ vanishes, which leads to \cref{udh_const_r-d}. The convergence rate of $u_{\delta}^{h}$
is stable around the first order (see \cref{table:Exam1:ac:exactcap:energynorm}), which confirms \cref{udh_const_r-d}. Although the convergence rates in \cref{table:Exam2:d2:mcons2,table:Exam2:d3:mcons2} are greater than the first order, the errors are significantly larger than that in \cref{Example:Continuous}. So, it can be reasonably argued that for the RHS \cref{rhs_pert} together with VCs \cref{dsquare,dtriple}, $E_{1}$ dominates the total error in the first few steps, and then $E_{2}$ takes over (it could be understood that $C_{7}$ is larger than $C_{8}$). Unfortunately, it is rather computationally demanding to carry out the last step in \cref{table:Exam2:d2:mcons2,table:Exam2:d3:mcons2}, that is $\delta_{0}/\delta=2^{5}$. To validate the explanation above, we turn to the simpler 1D counterpart.

\begin{example}\label{Example:Discontinuous1d}
We consider the nonlocal problem \cref{nonlocal_diffusion} on the domain $\Omega=(0,1)$
with the family of kernels $\{\gamma_{\delta}\}$ defined by \cref{Con_Kernel_d} in 1D case. Set RHS function 
\begin{equation}\label{rhs_pert_1d}
f_{\delta}(x)=-6x+\delta^2e^{x}
,\:\mbox{for}\:\: x\in\Omega,
\end{equation}
and the following two kinds of VCs
\begin{equation}\label{dsquare1d}
g_{\delta}(x) = x^{3}+\delta^2\sin(x),\:\mbox{for}\:\: x\in\Omdc,
\end{equation}
and
\begin{equation}\label{dtriple1d}
g_{\delta}(x) = x^{3}+\delta^3\sin(x),\:\mbox{for}\:\: x\in\Omdc,
\end{equation}
respectively. The solution of the corresponding local problem \cref{local_diffusion} with $f_{0}(x)=-6x$ and $g_{0}(x)=x^{3}$ is
\begin{equation*}
u_{0}(x)=x^{3}.
\end{equation*}
The \emph{exactcaps} solution is used to numerically
solve the nonlocal problem.
\end{example}
Let $\delta_{0}=0.3$, $m=3$, we use quasi-uniform meshes obtained from a randomly perturbed uniform mesh. To be specific, set $h=1/n_{1}$, then $\left\{x_{i}^{u}=ih:\:i=0,1,\cdots,n_{1}\right\}$ is a uniform mesh of $[0,1]$. The quasi-uniform mesh is obtained by adding a random vector ${\bf \varepsilon}\in\mathbb{R}^{n_{1}-1}$ (which obeys the uniform distribution on $[-0.2h,0.2h]$) to $x_{i}^{u}$ to reach $x_{i}^{u}+\varepsilon_i,\:i=1,2,\cdots,n_{1}-1$. Together with $x_{0}^{u}$ and $x_{n_{1}}^{u}$, the new mesh grids are constructed as follows
\begin{equation}\label{nonuni_grid}
x_{i}=x_{i}^{u}+\varepsilon_i,\:i=1,2,\cdots,n_{1}-1,\: x_{0}=x_{0}^{u},\: x_{n_{1}}=x_{n_{1}}^{u}.
\end{equation}
We have done over twenty tests of different random perturbations, and the convergence rates are all similar. Thus, instead of listing all of them, we select one test to verify our theoretical analysis. \cref{table:Exam3:EnNorm} provides errors and
convergence rates of $u_{\delta}^{h}$ for RHS \cref{rhs_pert_1d} together with VCs \cref{dsquare1d,dtriple1d} in the energy norm. It is seen that, in the first six steps, the convergence results are similar to the counterpart of the 2D case (see
\cref{table:Exam2:d2:mcons2,table:Exam2:d3:mcons2}). And then, as we expected earlier, the convergence rates approach the first order gradually. We also supply the error in the energy norm of the linear CDG approximation with an exact RHS and VC in boldface for the remaining seven steps. Since the error of $u_{\delta}^{h}$ with RHS \cref{rhs_pert_1d} and VC \cref{dtriple1d} is very close to that with the exact ones at the seventh step, its convergence rate approaches the first order there. While for the rate with the RHS \cref{rhs_pert_1d} and VC \cref{dsquare1d}, it takes more steps to reach the first order.

\begin{table}[ht]
\caption{\cref{Example:Discontinuous1d}: errors in the energy norm and convergence rates for \emph{exactcaps} solution}\label{table:Exam3:EnNorm} \footnotesize
\begin{tabular}{|c|cccc|c|cccc|c|}  
\hline
$\delta_{0}/\delta$ & \cref{dsquare1d} & Rate & \cref{dtriple1d} & Rate
&$2^{6}$ &3.41e-4&1.49&7.95e-5&0.96&$\mathbf{7.94e}$-$\mathbf{5}$ \\
\hline                                                                                                      
$2^{0}$ &2.06e-1& -- &9.79e-2& -- &$2^{7}$ &1.25e-4&1.45&3.97e-5&1.00&$\mathbf{3.96e}$-$\mathbf{5}$ \\
\hline                                                                                            
$2^{1}$&6.62e-2&1.64&1.90e-2&2.37&$2^{8}$ &4.58e-5&1.45&1.93e-5&1.04&$\mathbf{1.93e}$-$\mathbf{5}$ \\
\hline                                                                                            
$2^{2}$&2.25e-2&1.56&3.98e-3&2.25&$2^{9}$ &1.75e-5&1.39&9.62e-6&1.00&$\mathbf{9.63e}$-$\mathbf{6}$ \\
\hline                                                                                            
$2^{3}$&7.72e-3&1.54&1.01e-3&1.98&$2^{10}$&7.11e-6&1.30&4.92e-6&0.97&$\mathbf{4.92e}$-$\mathbf{6}$ \\
\hline                                                                                            
$2^{4}$&2.69e-3&1.52&3.46e-4&1.55&$2^{11}$&3.04e-6&1.23&2.44e-6&1.01&$\mathbf{2.40e}$-$\mathbf{6}$ \\
\hline                                                                                            
$2^{5}$&9.54e-4&1.49&1.55e-4&1.16&$2^{12}$&1.38e-6&1.14&1.22e-6&1.01&$\mathbf{1.20e}$-$\mathbf{6}$ \\
\hline
\end{tabular}
\end{table}

\begin{table}[ht]
\centering
\caption{Error estimate in the energy norm and implementation issue for the linear CDG solutions, $\lambda=2$}\label{table:summary} \footnotesize
\begin{tabular}{|c|c|c|c|}
\hline 
Interaction neighborhood  & $\left\|u_{\sharp}-\widetilde{u}_{0}\right\|_{\sharp}$ (Con) & $\left\|u_{\sharp}^{h}-\widetilde{u}_{0}\right\|_{\sharp}$ (Dis)  &  Implementation cost  \\    
\hline                                                               
 Euclidean ball $\sharp=\delta$& $\delta^{2}$     & $\delta^{2}+\delta^{-1}h^{2}$ & Demanding   \\
 Symmetric polygon $\sharp$=$(\delta|n_{\delta})$  & $\delta^{2}+n_{\delta}^{-2}$ & --    & Very demanding   \\
 Polygon $\sharp=(\delta,n_{\delta})$ & $\delta^{2}+\delta^{-1}n_{\delta}^{-2}$& $\delta^{2}+\delta^{-3}h^{2}$ & Less demanding    \\
\hline
\end{tabular}
\end{table}

\section{Concluding remarks}\label{sec:concl}
In this work, we estimated, in both energy and $L^{2}$ norms, the errors between the linear CDG solutions for some linear nonlocal problems and the solution of the local limit, simultaneously with respect to the horizon parameter and mesh size. Let us summarize in \cref{table:summary} the error estimates of the two linear CDG solutions in the energy norm, along with their implementation costs. Here (Con) stands for the continuum level, while (Dis) stands for the discrete level. For the case $\delta|n_{\delta}$ since it lacks an inner product and the induced norm, the corresponding error estimate on the discrete level is not given while the error estimate on the continuous level is actually measured in $\|\cdot\|_{\delta}$ norm.

\subsection{Other numerical methods}
Besides the CDG method, error estimates for other types of numerical solutions of nonlocal problems (like mesh-free method, collocation method, quadrature-based finite difference method) against the exact local solution may also be carried out in two steps like in this paper. Step 1 (on the continuum level): the error estimate of the nonlocal solutions with different interaction neighborhoods against the local exact solution, which is almost the same as the derivation in \cref{sec:Conv_d}. Step 2 (on the discrete level): the error estimate of the numerical solutions against the nonlocal exact solution removing the impact by the approximation of interaction neighborhood, which plays the same role as the conforming DG method in \cref{sec:Conv_h}. It should be noted that one does not always need to follow the same 2-step process here; for example in \cite{du2016asymptotically}, a different 2-step process has been given for Fourier spectral methods of nonlocal Allen-Cahn equation (1D in space). However, for numerical analysis in higher dimensions and with polygonal approximation to the original interaction neighborhood (Euclidean ball), our 2-step analysis could be more applicable.

\subsection{Other nonlocal problems}
We have focused on nonlocal problems with $L^1$ kernels and Dirichlet-type VCs and piecewise smooth data. The approach can be extended to Neumann or other types of VCs. 
One could consider more general kernels that might not be $L^1$. Theoretically,
in such cases, nonlocal problems with nonhomogeneous boundary data can be studied by
utilizing analytical findings given in \cite{du2022nonlocal}. The discussion on the effect of quadrature on the interaction neighborhoods will be more demanding due to potential singularities of the kernels used, although the modifications to the interaction neighborhoods are done generally away from such singularities. In this sense, we expect similar studies can be carried out. Furthermore, one might study extensions to other nonlocal problems, both nonlocal variational problems and nonlocal dynamical systems for which issues like the convergence of the nonlocal numerical solutions to the exact local continuum limit have also been considered either theoretically or numerically \cite{glusa2023asymptotically,jha2021finite,luo2023asymptotically,pasetto2022efficient,trask2019asymptotically,you2020asymptotically,yu2021asymptotically}.

\bibliographystyle{amsplain}
\bibliography{references}
\end{document}